\documentclass[a4paper,11pt]{article}
\usepackage[utf8]{inputenc} 
\usepackage[english]{babel}
\usepackage[T1]{fontenc}   
\usepackage{mathpazo}
\usepackage{graphicx}
\usepackage{comment}
\usepackage{amsthm,amsmath,amsfonts,stmaryrd,mathrsfs,amssymb}
\usepackage{mathtools}
\usepackage{bbm}
\usepackage{enumitem}
\usepackage{chemarrow}
\usepackage{tikz}
\usepackage[top=2.5cm, bottom=2.5cm, left=2cm, right=2cm]{geometry}
\usepackage{subfig}
\usepackage{multirow}
\usepackage{array}
\usepackage{bigints}
\usepackage{etoolbox}
\usepackage{mdframed}
\usepackage{color}
\usepackage[unicode,colorlinks=true,linkcolor=blue,citecolor=red,pdfencoding=auto,psdextra]{hyperref}
\setenumerate[1]{label=(\roman*), font = \normalfont}

\makeatletter
\newcommand{\mylabel}[2]{#2\def\@currentlabel{#2}\label{#1}}
\makeatother

\makeatletter

\@addtoreset{equation}{section}
\makeatother

\let\originalleft\left
\let\originalright\right
\renewcommand{\left}{\mathopen{}\mathclose\bgroup\originalleft}
\renewcommand{\right}{\aftergroup\egroup\originalright}

\newcommand{\ensemblenombre}[1]{\mathbb{#1}}
\newcommand{\N}{\ensemblenombre{N}}
\newcommand{\Z}{\ensemblenombre{Z}}

\newcommand{\R}{\ensemblenombre{R}}

\renewcommand{\P}{\mathbb{P}}
\newcommand{\E}{\mathbb{E}}

\newcommand{\Ec}[1]{\mathbb{E} \left[#1\right]}

\newcommand{\Pp}[1]{\mathbb{P} \left(#1\right)}

\newcommand{\Ecsq}[2]{\mathbb{E} \left[#1\mathrel{}\middle|\mathrel{}#2\right]}

\newcommand{\Ppsq}[2]{\mathbb{P} \left(#1\mathrel{}\middle|\mathrel{}#2\right)}

\newcommand{\intervalle}[4]{\mathopen{#1}#2
	\mathclose{}\mathpunct{},#3
	\mathclose{#4}}
\newcommand{\intervalleff}[2]{\intervalle{[}{#1}{#2}{]}}

\newcommand{\intervallefo}[2]{\intervalle{[}{#1}{#2}{)}}
\newcommand{\intervalleoo}[2]{\intervalle{(}{#1}{#2}{)}}
\newcommand{\intervalleentier}[2]{\intervalle\llbracket{#1}{#2}
	\rrbracket}
       
\newcommand{\petito}[1]{o\mathopen{}\left(#1\right)}

\newcommand{\grandO}[1]{O\mathopen{}\left(#1\right)}

\newcommand{\enstq}[2]{\left\lbrace#1\mathrel{}\middle|\mathrel{}#2\right\rbrace}

\newcommand{\ind}[1]{\mathbb{1}_{\lbrace #1 \rbrace}}  
\newcommand{\1}{\mathbb{1}}  

\newcommand{\cT}{\mathcal{T}}

\newcommand{\ttt}{\mathtt{t}}
\newcommand{\ttu}{\mathtt{u}}
\newcommand{\ttv}{\mathtt{v}}
\newcommand{\ttD}{\mathtt{D}}

\newcommand{\ttT}{\mathtt{T}}
\newcommand{\ttP}{\mathtt{P}}

\DeclareMathOperator{\haut}{ht}
\DeclareMathOperator{\diam}{diam}

\DeclareMathOperator{\wrt}{WRT}
\DeclareMathOperator{\pa}{PAT}
\DeclareMathOperator{\cst}{cst}

\DeclareMathOperator{\Var}{Var}
\newcommand{\diff}{\mathop{}\mathopen{}\mathrm{d}}
\newcommand{\abs}[1]{\left\lvert #1 \right\rvert}

\newmdtheoremenv{theorem}{Theorem}[section]
\newtheorem{proposition}[theorem]{Proposition}
\newtheorem{lemma}[theorem]{Lemma}

\newtheorem{corollary}[theorem]{Corollary}
\newtheorem{remark}[theorem]{Remark}

\begin{document}
\title{Correction terms for the height of weighted recursive trees}
\author{Michel Pain\footnote{Courant Institute of Mathematical Sciences, New York University.} \and Delphin Sénizergues\footnote{Department of Mathematics, University of British Columbia.}}
\maketitle

\begin{abstract}
	Weighted recursive trees are built by adding successively vertices with predetermined weights to a tree: each new vertex is attached to a parent chosen randomly proportionally to its weight.
	Under some assumptions on the sequence of weights, the first order for the height of such trees has been recently established in \cite{senizergues2019} by one of the authors. 
	In this paper, we obtain the second and third orders in the asymptotic expansion of the height of weighted recursive trees, under similar assumptions.
	Our methods are inspired from those used to prove similar results for branching random walks. 
	Our results also apply to a related model of growing trees, called the preferential attachment tree with additive fitnesses.
\end{abstract}

\section{Introduction}

Models of growing random trees have been widely studied for their connections with algorithms~\cite{devroye1998} and networks~\cite{ergun2002}; they have been used to model, among others, epidemic spread~\cite{moon1974} and pyramid schemes~\cite{gastwirthbhattacharya1984}.
See the survey \cite{smythemahmoud1994} and the book \cite{drmota2009} for a review of the literature.
In this paper, we consider a large family of such models that generalizes some well-studied cases, such as the uniform recursive tree or the plane oriented recursive tree, whose study dates back at least to~\cite{narapoport1970} and~\cite{szymanski1987} respectively.
For these simpler models, the first order for the height has been proved by Pittel~\cite{pittel1994} and the second and third orders in the asymptotic expansion can be deduced from similar results for the maximum of branching random walks.
The models of trees that we consider here can be seen as inhomogeneous versions of these simpler ones.
The first order for their height has been obtained recently in~\cite{senizergues2021} by one of the authors, and we prove here that the second and third orders are still similar to those appearing in the maximum of branching random walks, even though no direct connection can be used in this case.
We first present our model and results, and then discuss in more details some related works from the literature,
as well as the link between our model and branching random walks.

\subsection{Definition of the model and assumptions}

\paragraph{Definition of WRTs}
We define the model of weighted recursive trees, first introduced in \cite{borovkov2006} by Borovkov and Vatutin. 
For any sequence of non-negative real numbers $(w_n)_{n\geq 1}$ with $w_1>0$, we define the distribution $\wrt((w_n)_{n\geq 1})$ on sequences of growing rooted trees, which is called the \emph{weighted recursive tree with weights $(w_n)_{n\geq 1}$}. 
We construct a sequence of rooted trees $(\ttT_n)_{n\geq 1}$ starting from $\ttT_1$ containing only one root-vertex $\ttu_1$ and let it evolve in the following manner: the tree $\ttT_{n+1}$ is obtained from $\ttT_n$ by adding a vertex $\ttu_{n+1}$ with label $n+1$. The father of this new vertex is chosen to be the vertex with label $K_{n+1}$, where
\begin{align*}
\forall k\in \{1,\dots,n\}, \qquad \Ppsq{K_{n+1}=k}{\ttT_n}\propto w_k.
\end{align*}

Whenever we have any sequence of real numbers $(x_n)_{n\geq 1}$, we write $\boldsymbol x=(x_n)_{n\geq 1}$ in a bold font as a shorthand for the sequence itself, and $(X_n)_{n\geq 1}$ with a capital letter to denote the sequence of partial sums defined for all $n\geq 1$ as $X_n\coloneqq \sum_{i=1}^nx_i$. In particular, we do so for sequences of weights $(w_n)_{n\geq 1}$, so that $W_n$ always denotes the sum of the $n$ first weights. 
Some of our assumptions are expressed using the Landau big-O notation: we write $x_n=\grandO{y_n}$ if there exists a constant $C$ such that $\abs{x_n}\leq C \abs{y_n}$ for all $n\geq 1$.

\paragraph{Assumptions}

We assume that we work with a sequence $\boldsymbol w$ which satisfies the following assumption for some $\gamma > 0$, 
\begin{align} \tag{$\mathcal H_{1,\gamma}$} \label{eq:assumption_1}
\exists \lambda >0,\ \exists\alpha\in\intervalleoo{0}{1} : 
W_n = \lambda \cdot n^\gamma + O \left( n^{\gamma-\alpha} \right),
\end{align}
as $n \to \infty$.
Moreover, we assume in parts of the paper that
\begin{align} \tag{$\mathcal H_2$} \label{eq:assumption_2}
\sum_{i=n}^\infty \left( \frac{w_i}{W_i} \right)^2
= O \left( \frac{1}{n} \right).
\end{align}
Associated to the constant $\gamma$, we define another constant $\theta > 0$ as the unique positive solution to the following equation
\begin{align}\label{eq:definition_theta}
1+\gamma\left(e^\theta-1-\theta e^\theta \right) = 0.
\end{align}
Under assumption \eqref{eq:assumption_1} on the sequence of weights $(w_n)_{n\geq 1}$, it was shown in \cite{senizergues2021} that the height of the tree satisfies 
\begin{align}\label{eq:asymptotic for the height first order}
	\frac{\haut(\ttT_n)}{\log n}\underset{n \rightarrow\infty}{\longrightarrow} \gamma e^{\theta}
\end{align}
almost surely. 

\subsection{Main results}
Our results consist in computing the next order terms for the asymptotic behaviour \eqref{eq:asymptotic for the height first order}, which contains a logarithmic correction followed by a term of constant order. This is contained is the following theorem.
\begin{theorem} \label{thm:height_WRT}
Under assumptions \eqref{eq:assumption_1} and \eqref{eq:assumption_2}, the following sequence of random variables is tight
\[
	\left( \haut(\ttT_n) - \gamma e^{\theta} \log n + \frac{3}{2 \theta} \log \log n \right)_{n \geq 2}.
\]
\end{theorem}
In the case of the upper bound for the height, we have a more precise result, requiring only assumption \eqref{eq:assumption_1}, which gives an explicit bound for the tail distribution of the height.
This bound should be optimal up to the value of the constant $C=C(\boldsymbol w)$.
\begin{theorem} \label{thm:upper_bound_height_WRT}
Under assumption \eqref{eq:assumption_1}, there exists $C > 0$ such that for any $n,x \geq 1$,
\[
	\Pp{\haut(\ttT_n) \geq \gamma e^{\theta} \log n - \frac{3}{2 \theta} \log \log n + x}
	\leq C x e^{-\theta x}.
\]
\end{theorem}
The next theorem ensures that 
the set of vertices in $\ttT_n$ having height close to $\haut(\ttT_n)$
are not all close parents, meaning that some of them have a most recent common ancestor that is of height of constant order. 
This has the effect that the diameter of the tree is close to twice its height (which is an obvious upper-bound for the diameter). 
This is stated in the following theorem.
\begin{theorem} \label{thm:diameter_WRT}
Under assumptions \eqref{eq:assumption_1} and \eqref{eq:assumption_2}, the following sequence of random variables is tight
\[
	\left( \diam(\ttT_n) - 2\gamma e^{\theta} \log n + \frac{3}{\theta} \log \log n \right)_{n \geq 2}.
\]
\end{theorem}

\paragraph{The case of i.i.d. weights} 
A natural setting to consider is to consider the case where we run the model starting with an i.i.d.\@ random sequence of weights $(\mathsf w_n)_{n\geq 1}$, say with law $\mu$ on $\intervalleoo{0}{\infty}$. 
In this case it is quite easy to check that, if $\mu$ admits a moment of order $2$, then the random sequence $(\mathsf w_n)_{n\geq 1}$ almost surely satisfies \eqref{eq:assumption_1} with $\gamma=1$, and also \eqref{eq:assumption_2}. Remark that the value of $\theta$ associated to $\gamma=1$ by \eqref{eq:definition_theta} is $\theta=1$. 
This directly allows to apply Theorem~\ref{thm:height_WRT} and Theorem~\ref{thm:diameter_WRT} in this setting.

If $\mu$ only has a moment of order $1+\epsilon$ for some positive $\epsilon$, then we still have the fact that the random sequence $(\mathsf w_n)_{n\geq 1}$ almost surely satisfies \eqref{eq:assumption_1} with $\gamma=1$. 
In this case, we get that the result of Theorem~\ref{thm:upper_bound_height_WRT} holds conditionally on the sequence $(\mathsf w_n)_{n\geq 1}$. 
Integrating this over the sequence $(\mathsf w_n)_{n\geq 1}$ entails that at least the upper-bound in Theorem~\ref{thm:height_WRT} is true, \emph{i.e}
\begin{equation}
\sup _{n\geq 1}\Pp{\haut(\ttT_n) \geq  e \log n - \frac{3}{2} \log \log n + b} \underset{b\rightarrow \infty}{\longrightarrow}0.
\end{equation}
We remark that this statement in the case of random weights is weaker than the one for deterministic weights, as the speed of the convergence to $0$ is not explicit here. 
In the statement of Theorem~\ref{thm:upper_bound_height_WRT}, the constant $C$ appearing on the right-hand side depends on the sequence of weights in a non-explicit way, and getting the same tail bound as in Theorem~\ref{thm:upper_bound_height_WRT} would require to integrate the value of this non-explicit function over the law of the sequence $(\mathsf w_n)_{n\geq 1}$.

\subsection{Application to preferential attachment trees}
We introduce here another family of growing trees and explain how to apply the results of Theorem~\ref{thm:height_WRT}, Theorem~\ref{thm:upper_bound_height_WRT} and Theorem~\ref{thm:diameter_WRT} to this other setting.

\paragraph{Definition of PATs}
We define a process on growing random trees called the \emph{preferential attachment tree with additive fitnesses}, or $\pa$ for short.
This model depends on a sequence $\mathbf{a}=(a_i)_{i\geq 1}$ of non-negative numbers, which represent the initial fitnesses of the vertices. 
For non-constant sequences $\mathbf{a}$, this model was introduced for the first time in \cite{ergun2002} by Erg\"un and Rodgers. 
As before, we iteratively construct a sequence of rooted trees $(\ttP_n)_{n\geq 1}$ starting from $\ttP_1$ containing only one root-vertex $\ttu_1$ labelled $1$, and evolving in the following manner.
The tree $\ttP_{n+1}$ is obtained from $\ttP_n$ by adding a vertex $\ttu_{n+1}$ with label $n+1$. 
The father of this new vertex is chosen to be the vertex with label $J_{n+1}$, where
\begin{align*}
\forall k\in \{1,\dots,n\}, \qquad \Ppsq{J_{n+1}=k}{\ttP_n}\propto \deg_{\ttP_n}^+(\ttu_k)+a_k,
\end{align*}
where $\deg_{\ttP_n}^+(\ttu_k)$ denotes the out-degree of $\ttu_k$ in the tree $\ttP_n$. 	
In the particular case where $n = 1$, we set $J_2=1$, even in the case $a_1=0$ for which the last display does not make sense.

\paragraph{Connection with WRTs with a random sequence of weights}

First recall that, for $a,b>0$, the distribution $\mathrm{Beta}(a,b)$ has density $\frac{\Gamma(a+b)}{\Gamma(a)\Gamma(b)}\cdot x^{a-1}(1-x)^{b-1} \cdot \ind{0\leq x\leq 1}$ with respect to Lebesgue measure.
If $b=0$ and $a>0$, we use the convention that the distribution $\mathrm{Beta}(a,b)$ is a Dirac mass at $1$.

Now, \cite[Theorem~1.1]{senizergues2021} tells us the following. 
For any sequence $\mathbf a$ of fitnesses, we define the associated random sequence $\boldsymbol{\mathsf w}^\mathbf a=(\mathsf{w}^\mathbf{a}_n)_{n\geq 1}$ through its corresponding partial sums $\mathsf{W}^\mathbf{a}_n=\sum_{k=1}^{n}\mathsf{w}^\mathbf{a}_n$ as 
\begin{equation}\label{def:random weight associated to sequence of fitnesses}
\mathsf{w}^\mathbf{a}_1=\mathsf{W}^\mathbf{a}_1 \coloneqq 1 
\qquad \text{and} \qquad 
\forall n\geq 2, \quad \mathsf{W}^\mathbf{a}_n \coloneqq \prod_{k=1}^{n-1}\beta_k^{-1},
\end{equation}
where the $(\beta_k)_{k\geq 1}$ are independent with respective distribution $\mathrm{Beta}(A_k+k,a_{k+1})$, and $A_k \coloneqq \sum_{i=1}^{k}a_i$. 
Then,
the distributions $\pa(\mathbf a)$ and $\wrt(\boldsymbol{\mathsf w}^\mathbf a)$ coincide.

The strategy to apply our results to preferential attachment trees is to use this connection and to check that under some assumptions on the sequence $\mathbf a$, the corresponding random sequence of weights $(\mathsf{w}^\mathbf{a}_n)_{n\geq 1}$ almost surely satisfies the assumptions of our theorems. 

\paragraph{Almost sure behaviour of \texorpdfstring{$(\mathsf{w}^\mathbf{a}_n)_{n\geq 1}$}{(wna)}}
We assume here that the sequence of fitnesses $\mathbf a=(a_i)_{i\geq1}$ satisfies
\begin{align}\label{eq:assumption_1 fitness sequence}\tag{$\mathcal H_{1,\zeta}^{\mathrm{PAT}}$}
A_n\coloneqq  \sum_{i=1}^na_i=\zeta\cdot n + \grandO{n^{1-\delta}},
\end{align}
for some $\zeta>0$ and some $\delta>0$. 
Then \cite[Proposition~1.3]{senizergues2021} tells us that under this assumption for the sequence $\mathbf a$, the random sequence $(\mathsf{w}^\mathbf{a}_n)_{n\geq 1}$ almost surely satisfies \eqref{eq:assumption_1} with $\gamma=\frac{\zeta}{\zeta+1}$.
This allows us to apply Theorem~\ref{thm:upper_bound_height_WRT} to preferential attachment trees with any sequence of fitnesses satisfying \eqref{eq:assumption_1 fitness sequence}, and obtain the following corollary.

\begin{corollary}\label{cor:upper bound on the height for PAT}Under assumption \eqref{eq:assumption_1 fitness sequence} for the sequence of fitnesses $\mathbf{a}$, we have
	\[\sup_{n\geq 1} \Pp{\haut(\ttP_n)\geq \gamma e^\theta \log n -\frac{3}{2\theta} \log\log n +b}\underset{b \rightarrow\infty }{\longrightarrow} 0,\]
	with $\gamma=\frac{\zeta}{\zeta+1}$ and $\theta$ defined from $\gamma$ as in \eqref{eq:definition_theta}.
\end{corollary}

In order to also get the lower bound given by Theorem~\ref{thm:height_WRT}, we need to assume some additional hypothesis on the sequence $\mathbf a$, namely
\begin{align}\label{eq:assumption_2 fitness sequence}\tag{$\mathcal H_2^{\pa}$}
\sum_{i=1}^{n}a_i^2=\grandO{n}.
\end{align}
The following lemma, proved in the appendix, then ensures that under \eqref{eq:assumption_1 fitness sequence} and \eqref{eq:assumption_2 fitness sequence}, the random sequence $(\mathsf{w}^\mathbf{a}_n)_{n\geq 1}$ almost surely  satisfies \eqref{eq:assumption_2}, so that the assumptions of Theorem~\ref{thm:height_WRT} and Theorem~\ref{thm:diameter_WRT} are satisfied.   
\begin{lemma}\label{lem:remainder sum p_i square for PAT}
	If the sequence $\mathbf{a}$ satisfies \eqref{eq:assumption_1 fitness sequence} and \eqref{eq:assumption_2 fitness sequence}, then almost surely 
	\begin{align*}
	\sum_{i=n}^{\infty}\left(\frac{\mathsf{w}^\mathbf{a}_i}{\mathsf{W}^\mathbf{a}_i}\right)^2 = \grandO{n^{-1}}.
	\end{align*}
\end{lemma}
This allows us to get the following analog of Theorem~\ref{thm:height_WRT} and Theorem~\ref{thm:diameter_WRT} in the context of preferential attachment trees.
\begin{corollary}\label{cor:height and diameter for PAT}Under the assumptions \eqref{eq:assumption_1 fitness sequence} and \eqref{eq:assumption_2 fitness sequence} for the sequence of fitnesses $\mathbf{a}$, the sequences
	\[\left(\haut(\ttP_n)- \gamma e^\theta \log n +\frac{3}{2\theta} \log\log n\right)_{n\geq 1}\]
	and 
	\[\left(\diam(\ttP_n)- 2\gamma e^\theta \log n +\frac{3}{\theta} \log\log n\right)_{n\geq 1}\]
	are tight, where $\gamma=\frac{\zeta}{\zeta+1}$ and $\theta$ is defined from $\gamma$ as in \eqref{eq:definition_theta}.
\end{corollary}

\paragraph{The case of i.i.d. fitnesses} As for the case of WRTs, a natural model is to start from a sequence $\mathbf{a}=(a_i)_{i\geq 1}$ that is i.i.d.\ with some distribution $\mu$ over $\intervallefo{0}{\infty}$ that is not concentrated on $0$.
From the discussion above, we see that if $\mu$ has a moment of order $1+\epsilon$ for some positive $\epsilon$, then \eqref{eq:assumption_1 fitness sequence} holds almost surely and Corollary~\ref{cor:upper bound on the height for PAT} applies where $\zeta$ is the first moment of $\mu$. 
If furthermore $\mu$ has a second moment, then $\mathbf a$ satisfies \eqref{eq:assumption_2 fitness sequence} almost surely and so, thanks to Lemma~\ref{lem:remainder sum p_i square for PAT}, the conclusions of Corollary~\ref{cor:height and diameter for PAT} hold.

\subsection{Related works and comments}

An asymptotic expansion for the height of recursive trees identical to Theorem~\ref{thm:height_WRT} has already been obtained for some specific models.
The first result of this type has been shown by Drmota~\cite{drmota2003} and Reed~\cite{reed2003} for  binary search trees: these trees are a sequence of random subtrees of the infinite binary tree, recursively built by adding new vertices uniformly at random among all the possible sites.
Note that these trees do not enter in the framework of WRTs.
The simplest WRT is the uniform recursive tree, obtained by taking all weights equal to 1.
In this case, the asymptotic expansion has been obtained by Addario-Berry and Ford \cite{addario-berryford2013}.
Slight modifications of the uniform recursive tree have also been covered: Hoppe trees, where all weights except $w_1$ equal 1, have been studied in \cite{leckeyneininger2013}, and another extension, where finitely many weights are different from 1, in \cite{hiesmayrislak2020}.

The asymptotic expansion in Theorem~\ref{thm:height_WRT} is also similar to the one for the maximal position in a branching Brownian motion \cite{bramson1978,bramson1983} or a branching random walk \cite{hushi2009,addario-berryreed2009,aidekon2013}. 
More generally, this behaviour for the maximum is shared by the universality class of log-correlated fields, see \cite{arguin2017} for a review.
For this large class of models, the maximum should behave asymptotically as
\begin{equation}\label{eq:general asymptotic expansion}
v \log n - \frac{3}{2 \beta_c} \log \log n+ O(1),
\end{equation}
where $n$ is the number of particles involved, $v$ is a constant depending on the model and $\beta_c$ is the critical inverse temperature of the system.
The fact that in our case $\beta_c = \theta$ can be seen from the fact that $\theta$ is the smallest real number such that a vertex $\ttu_i$ chosen in $\ttT_n$ proportionally to $\frac{w_i}{W_n} e^{\theta \haut(\ttu_i)}$ has a height asymptotically equivalent to $\haut(\ttT_n)$, see Lemma~\ref{lem:many_to_one} and \cite{senizergues2021}.
Furthermore, note that the precise upper tail in Theorem~\ref{thm:upper_bound_height_WRT} is known to be optimal for branching random walks, up to the value of the constant.

Connections between recursive trees and branching processes have been widely used since the works of Pittel \cite{pittel1984,pittel1994} and Devroye \cite{devroye1986,devroye1987,devroye1998}.
In particular the height of the uniform recursive tree can be deduced from the counterpart for branching random walks as follows. Consider a continuous-time branching random walk, starting with one particle at position 0 at time 0
and where each particle lives during an exponential time with parameter 1, during which it stays at position $h$ where it was born, and then splits into two particles at positions $h$ and $h+1$.
If $\tau_n$ denotes the first time where $n$ particles are alive in this branching random walk, then the distribution of the positions of the particles at time $\tau_n$ is the same as the distribution of the heights of vertices in the uniform recursive tree $\ttT_n$.
Since $\tau_n = \log n + O(1)$ in probability, the asymptotic development in Theorem~\ref{thm:height_WRT} for the uniform recursive tree follows directly from the result of Aïdékon \cite{aidekon2013}. 
Note that Addario-Berry and Ford \cite{addario-berryford2013} used a different connection with branching random walks to prove the asymptotic expansion for the height of the uniform recursive tree.

It is important to note that in the case of a general WRT, we cannot directly deduce Theorem~\ref{thm:height_WRT} from the result for branching random walks.
We can still link the tree $\ttT_n$ to an \textit{inhomogeneous} continuous-time branching random walk defined as follows. 
Let $\mathrm{Exp}(\lambda)$ denote the exponential distribution with parameter $\lambda$.
We start with one particle at position $0$ at time $0$ with an $\mathrm{Exp}(w_1)$ lifetime.
When a particle at position $h$ with an $\mathrm{Exp}(w_i)$ lifetime dies, it split into two particles, one at $h$ with an $\mathrm{Exp}(w_i)$ lifetime and another at $h+1$ with an $\mathrm{Exp}(w_k)$ lifetime, if this is the $(k-1)$-th death event in the whole process.
Then, with $\tau_n$ defined as before, the distribution of the positions of the particles at time $\tau_n$ is the same as the distribution of the heights of the vertices in the tree $\ttT_n$ with distribution $\wrt((w_n)_{n\geq 1})$.
However, in addition to being inhomogeneous, the branching random walk defined here does not satisfy the branching property: the progeny of a particle depends on the progeny of the other particles alive at the same time. 
Consequently, results from the literature cannot be directly applied to this model.
Nonetheless, we managed to adapt the methods used to prove asymptotics for the maximum of branching random walks (e.g.\@ in \cite{aidekon2013}) directly in context of weighted recursive trees, see Section~\ref{subsection:organization} for an overview of the proof.

In the case of preferential attachment trees with additive fitnesses with a constant sequence $\mathbf a$, 
we believe that the same type of comparison with a branching random walk as above could lead directly to the asymptotic expansion \eqref{eq:general asymptotic expansion}; however, we have failed to find a reference for that fact in the literature.
For non-constant sequences $\mathbf{a}$, deterministic or i.i.d., this connection would break down 
and obtaining such an asymptotic expansion would again not straightforwardly follow from known results.

A natural question is the convergence in distribution of the height of weighted recursive trees after centering as in Theorem~\ref{thm:height_WRT}.
This convergence has been proved for branching Brownian motion \cite{bramson1983,lalleysellke1987} and for non-lattice branching random walks \cite{aidekon2013} and the limit is a randomly shifted Gumbel random variable.
In the case of a lattice branching random walk (such as the ones mentioned before), no general result has been established so far and one can only hope that the centered height oscillates around a non-universal limiting distribution \cite{lifshits2012,corre2017}.
For recursive trees, this convergence is known only for binary search trees: it has been shown by Drmota \cite{drmota2003}, via analytic methods, and by Corre~\cite{corre2017}, who uses the connection with a similar continuous-time branching random walk as the one above and gives a different description of the limit than that of Drmota.

Another future direction of study would be the case where \eqref{eq:assumption_1} is not satisfied, in particular where $W_n$ grows sub- or super-polynomially. This is not done in this paper, but we expect universality to break in these cases.

Last, we mention some other contributions about WRTs that investigate other properties than the height, under various assumptions for the behaviour of the sequence of weights $(w_n)_{n\geq 1}$. 
The model of WRT has been introduced by Borovkov and Vatutin in \cite{borovkov2006}, 
in which they study the asymptotic behaviour of the height of the $n$-th vertex, as well as some properties on the degree of vertices in the tree, under the assumption that the weights have a certain product form.
Recently, Mailler and Uribe Bravo \cite{maillerbravo2019} proved the convergence of the weighted profile of the tree to a Gaussian, in the sense of weak convergence, for a variety of random sequences $(w_n)_{n\geq 1}$ that exhibit a very wide range of asymptotic behaviours. 
Convergence of the profile in a strong sense is also proved in \cite{senizergues2021}, under assumptions that ensure that the weights behave more or less polynomially, similar to the ones in this paper. 
Also recently, Lodewijks and Ortgiese studied in \cite{lodewijks2020} a similar model of weighted random graphs (which contains the case of trees) under the assumption that $(w_n)_{n\geq 1}$ is i.i.d.\@ with some distribution $\mu$. 
Under a first moment assumption on $\mu$, they prove the convergence of the empirical distribution of the degrees and the weights of vertices in the graph. 
They also describe the behaviour of the maximal degree under several different assumptions for the tail of $\mu$. 
The convergence of the degree distribution in the i.i.d.\ setting can also be seen as a particular case of some results by Iyer in \cite{iyer2020} and by Fountoulakis, Iyer,  Mailler and Sulzbach in \cite{fountoulakis2019}, both times proved in a more general model of growing graphs.

\subsection{Overview of the paper}
\label{subsection:organization}

The paper is mainly dedicated to the proof of Theorem~\ref{thm:height_WRT} concerning the height of weighted recursive trees and this proof is split into two parts: the upper bound and the lower bound. For the upper bound, we actually prove Theorem~\ref{thm:upper_bound_height_WRT} which implies the upper bound in Theorem~\ref{thm:height_WRT}.
Theorem~\ref{thm:diameter_WRT}, concerning the diameter of the trees, is a byproduct of the proof of the lower bound in Theorem~\ref{thm:height_WRT}.
Concerning preferential attachment trees, the only result we need to show is Lemma~\ref{lem:remainder sum p_i square for PAT} and it is proved in Appendix~\ref{section:PAT}.

Our strategy is to adapt the methods used to prove the asymptotic expansion of the maximum of a branching random walk and, for this, we rely on the same basic tools: many-to-one and many-to-two lemmas.
These lemmas, established in Section~\ref{sec:many-to-few}, allow us to compute the first and second moment of quantities of the form
\[
\sum_{i=1}^{n} \frac{w_i}{W_n} \cdot e^{\theta\haut(\ttu_i)}\cdot F(\haut(\ttu_i(1)),\haut(\ttu_i(2)),\dots ,\haut(\ttu_i(n))),
\]
where $\ttu_i(k)$ denotes the closest ancestor of $\ttu_i$ in $\ttT_k$, and $F$ is a real-valued function. 
We call the sequence $(\haut(\ttu_i(1)),\dots ,\haut(\ttu_i(n)))$ the \emph{trajectory} of vertex $\ttu_i$ in the construction of $\ttT_n$.
The first moment of the quantity appearing in the last display is expressed in terms of $\Ec{F(H_1,H_2,\dots,H_n)}$, where $(H_i)_{i\geq 1}$ is a time-inhomogeneous random walk, whose step distributions depend on the $w_i$'s.
The expression for the second moment involves two random walks that coincide at the beginning of their trajectory and that are then only weakly dependent: this differs from the behaviour observed in branching random walks where the trajectories of two different particles are independent after their splitting point.
These lemmas rely on a coupling result from \cite{maillerbravo2019}, which describes a joint construction of the tree as well as two distinguished vertices in the tree, in a way that makes the trajectory of those vertices easy to analyze.

In Section~\ref{section:upper_bound}, we prove Theorem~\ref{thm:upper_bound_height_WRT}, which implies the upper bound in Theorem~\ref{thm:height_WRT}.
Its proof relies only on first moment calculations using the many-to-one lemma.
The first step is to prove that for $K$ large enough, with high probability, for any $n\geq1$ we have $\haut(\ttu_n) \leq \gamma e^\theta \log n +K$, and we then work on this event in order to prevent the first moment from blowing up.
The end of the argument is then close to the method used by Aïdékon \cite{aidekon2013} for branching random walks: we use a first moment calculation on the number of high vertices on the aforementioned event, dealing separately with vertices whose trajectory reach a high point too soon, each leading to a large cluster of high vertices.

The lower bound in Theorem~\ref{thm:height_WRT} is established in Section~\ref{section:lower_bound}. We use a first and second moment calculation on a well-chosen quantity $Q_n$, which is the total weight of sufficiently high vertices in $\ttT_n$ whose trajectory has stayed below an appropriate barrier (see \eqref{def:Q_n^(N)}).
For the branching random walk, this calculation usually shows that $\P(Q_n > 0) \geq c$ with $c$ a positive constant and one can conclude using the branching property: wait until there is a large number $N$ of particles alive and then each of these particles has a probability $c$ of having a very high descendant, independently of each other.
In our case, this second step of the argument is harder to justify: the subtrees rooted at the $N$ first vertices are not independent and do not necessarily satisfy our assumptions (some of them can even be finite). 
Therefore, we use a different approach to show directly that $\P(Q_n > 0) \to 1$ as $n \to \infty$. 
This can be shown via a first and second moment calculation on $Q_n$ only if typically the most recent common ancestor of two vertices contributing to $Q_n$ is the root.
To this end, we first choose a very constraining barrier so that the most recent common ancestor has to be typically in the first $O(1)$ vertices.
Then, we consider a modified tree $\ttT_n^{(N)}$, where we transfer the weights of the $N$ first vertices to the root. 
We actually do our calculation on this tree, for which the most recent common ancestor of two vertices contributing to $Q_n$ is the root with high probability when $N \to \infty$.
Since the height of $\ttT_n^{(N)}$ is stochastically dominated by the height of $\ttT_n$, this is sufficient.

In Section~\ref{section:diameter}, we prove Theorem~\ref{thm:diameter_WRT} showing that the diameter of the tree is twice its height, up to a $O(1)$ term. The upper bound is trivial and the lower bound follows from the fact that, in the tree $\ttT_n^{(N)}$, we can find with high probability two very high vertices whose most recent common ancestor is the root.

We also need precise estimates for the time-inhomogeneous random walks appearing in the many-to-one and many-to-two lemmas. These random walks have Bernoulli jumps with smaller and smaller parameters and therefore known results cannot be directly applied. 
In Section~\ref{section:RW}, we compare these random walks with a time-homogeneous random walk with Poisson jumps to establish these estimates. Note that this section has to be written in a relative generality, so that the same result can be applied in different contexts in the paper.

Throughout the paper, $C$ and $c$ denote positive constants that can only depend on the weights $(w_i)_{i\geq 1}$ and that can change from line to line. Typically, $C$ should be thought as sufficiently large and $c$ as sufficiently small.
For sequences $(a_n)_{n\geq1}$ and $(b_n)_{n\geq1}$ of real numbers, we say that $a_n = O(b_n)$ as $n \to \infty$ if there is a constant $C$, depending only on the weights $(w_i)_{i\geq 1}$, such that $\abs{a_n} \leq C \abs{b_n}$ for any $n \geq 1$.
Let $\N \coloneqq \{0,1,2,\dots\}$ and, for $a,b \in \R$, $\intervalleentier{a}{b} \coloneqq \intervalleff{a}{b} \cap \Z$.

\section{Distinguished points and many-to-few lemmas}
\label{sec:many-to-few}

\subsection{Some terminology}

\paragraph{Recursive trees}
Recursive trees on $n$ vertices are rooted trees whose vertices are labeled with the integers $1$ to $n$ such that the labels along any path starting from the root form a strictly increasing sequence.
We denote $\mathbb{T}_n$ the set of such trees. 
Note that the root is necessarily the vertex with label $1$. 
According to these definitions, the sequence $(\ttT_n)_{n\geq 1}$ constructed in the introduction takes its values in $\bigcup_{n\geq 1}\mathbb T_n$.
We also introduce
\begin{align*}
	\mathbb{T}^\bullet_n \coloneqq \enstq{(\ttt, \ttu)}{\ttt\in \mathbb{T}_n,\  \ttu \in \ttt}, \quad \text{and} \quad \mathbb{T}^{\bullet\bullet}_n \coloneqq \enstq{(\ttt,\ttu,\ttv)}{\ttt\in \mathbb{T}_n,\  \ttu \in \ttt,\ \ttv\in \ttt},
\end{align*}
the set of recursive trees of size $n$ endowed with respectively one or two distinguished vertices.

\paragraph{Labels and ancestors of a vertices} For any $(\ttt,\ttu)\in\mathbb{T}_n^{\bullet}$, we write $\mathrm{lab}(\ttu)$ for the label of vertex $\ttu$ in the tree $\ttt$, which is an integer between $1$ and $n$.
For any $k\leq n $ we write $\ttu(k)$ for the most recent ancestor of $\ttu$ that has label smaller than or equal to $k$. 
For any $(\ttt,\ttu,\ttv) \in \mathbb{T}^{\bullet\bullet}_n$, we denote $\ttu\wedge \ttv$ the most recent common ancestor of $\ttu$ and $\ttv$ in the tree $\ttt$.

\subsection{Model with two distinguished vertices}
We introduce here a very useful construction of the trees $(\ttT_n)_{n\geq 1}$ which is coupled with the choice of some distinguished vertices on those trees. It is due to Mailler and Uribe Bravo \cite[Section~2.4]{maillerbravo2019}.  
For $n\geq 2$, let $B_n$ and $\widetilde{B}_n$ be two independent Bernoulli random variables with parameter $\frac{w_{n}}{W_{n}}$, independent for all $n\geq 2$.
For $n\geq 1$, let $J_n$ be a random variable on $\{1,\dots, n\}$ such that $\Pp{J_n=k}=\frac{w_k}{W_n}$, also independent of all other random variables.
We define a sequence $((\ttT_n,\ttD_n,\widetilde{\ttD}_n))_{n\geq 1}$, where at each time $n\geq 1$ we have $(\ttT_n,\ttD_n,\widetilde{\ttD}_n)\in \mathbb T_n^{\bullet \bullet }$, by the following procedure.
\begin{itemize}
	\item The tree with distinguished vertex $(\ttT_1,\ttD_1,\widetilde{\ttD}_1)$ is the only recursive tree with one vertex and the vertices $\ttD_1$ and $\widetilde{\ttD}_1$ are equal to this vertex.
	\item At every step $n \geq 1$, conditionally on $(\ttT_n,\ttD_n,\widetilde{\ttD}_n)$,
	\begin{itemize}
		\item if $(B_{n+1},\widetilde{B}_{n+1})=(1,0)$, the tree $\ttT_{n+1}$ is obtained by attaching a new vertex $\ttu_{n+1}$ to the distinguished vertex $\ttD_n$, and setting $\ttD_{n+1}=\ttu_{n+1}$, and $\widetilde{\ttD}_{n+1}=\widetilde{\ttD}_n$,
		\item if $(B_{n+1},\widetilde{B}_{n+1})=(0,1)$, the tree $\ttT_{n+1}$ is obtained by attaching a new vertex $\ttu_{n+1}$ to the distinguished vertex $\widetilde{\ttD}_n$, and setting $\ttD_{n+1}=\ttD_{n}$, and  $\widetilde{\ttD}_{n+1}=\ttu_{n+1}$,
		\item if $(B_{n+1},\widetilde{B}_{n+1})=(0,0)$, the tree $\ttT_{n+1}$ is obtained by attaching a new vertex $\ttu_{n+1}$ to the vertex $\ttu_{J_n}$, and setting $\ttD_{n+1}=\ttD_{n}$, and  $\widetilde{\ttD}_{n+1}=\widetilde{\ttD}_n$,
		\item if $(B_{n+1},\widetilde{B}_{n+1})=(1,1)$, the tree $\ttT_{n+1}$ is obtained by attaching a new vertex $\ttu_{n+1}$ to the distinguished vertex $\ttD_n$, and setting $\ttD_{n+1}=\ttu_{n+1}$, and $\widetilde{\ttD}_{n+1}=\ttu_{n+1}$. 
	\end{itemize}
\end{itemize}
The following proposition is \cite[Proposition~9]{maillerbravo2019}, slightly rephrased for our purposes. 
\begin{proposition}{\cite[Proposition~9]{maillerbravo2019}}\label{prop: construction deux points distingués}
The sequence $(\ttT_n)_{n\geq 1}$ defined above has distribution $\wrt(\boldsymbol{w})$. 

Furthermore, for any $n\geq 1$, conditionally on $\ttT_n$, the points $\ttD_n$ and $\widetilde{\ttD}_n$ are sampled on $\ttT_n$ independently with distribution $\mu_n$, where $\mu_n$ is the probability measure supported on $\ttT_n$ such that for any $1\leq k \leq n$ we have $\mu_n(\{\ttu_k\})= \frac{w_k}{W_n}$. 
This entails that for any $n\geq 1$ and any function $\Phi \colon \mathbb{T}^{\bullet\bullet}_n \to \R$, we have
\begin{align}\label{eq:égalité en loi deux points distingués}
	\Ec{\Phi(\ttT_n,\ttD_n,\widetilde{\ttD}_n)}=\Ec{\sum_{1\leq i,j\leq n}\frac{w_i w_j}{W_n^2}\Phi(\ttT_n,\ttu_i,\ttu_j)}.
\end{align}
For a function $\Psi \colon \mathbb{T}^{\bullet}_n \to \R$, this can be re-written as
\begin{align}\label{eq:égalité en loi un point distingué}
\Ec{\Psi(\ttT_n,\ttD_n)}=\Ec{\sum_{1\leq i\leq n}\frac{w_i}{W_n}\Psi(\ttT_n,\ttu_i)}.
\end{align}
\end{proposition}

\paragraph{Remarks about the construction}
In the previous construction, we can remark that the sequence $(\ttD_n)_{n \geq 1}$ is non-decreasing in the genealogical order so that for any $k\leq n$ we have $\ttD_n(k)=\ttD_k$. This is not the case for $(\widetilde{\ttD}_n)_{n \geq 1}$. Also, we can write 
 \begin{align}\label{eq:hauteur point distingué est somme de bernoulli}
 	\forall 1\leq k \leq n, \quad  \haut(\ttD_n(k))= \haut(\ttD_k)= \sum_{i=2}^{k}B_i.
 \end{align}
Denoting $I_n \coloneqq \max \{1\leq k \leq n : B_k=\widetilde{B}_k=1 \}$ with the convention that $B_1=\widetilde{B}_1=1$ to make the last set non-empty,  we can also write
\begin{align*}
	\forall 1\leq k \leq n, \quad  \haut(\widetilde{\ttD}_n(k))=\begin{cases}  \sum_{i=2}^{k}B_i &\text{if } k\leq I_n\\
	 \sum_{i=2}^{I_n}B_i + \sum_{i=I_n+1}^{n}\widetilde{B}_i &\text{otherwise}
	\end{cases}
\end{align*} 
Note that $I_n$ is equal to $\mathrm{lab}(\ttD_n\wedge \widetilde{\ttD}_n)$, the label of the most recent common ancestor between $\ttD_n$ and $\widetilde{\ttD}_n$.

\subsection{Change of measures and many-to-one}

\paragraph{Change of measure}
For any $n\geq 1$ the tree with two distinguished vertices $(\ttT_n,\ttD_n,\widetilde{\ttD}_n)$ defined above only depends on the sequences $(B_i)_{2\leq i \leq n}$, $(\widetilde{B}_i)_{2\leq i \leq n}$ and $(J_i)_{1\leq i \leq n-1}$.
Recall $\theta > 0$ is defined by \eqref{eq:definition_theta}.
We can introduce $\P_{\theta}$ in such a way that 
\begin{align}
\frac{\diff \P_{\theta}}{\diff \P}
= \frac{\prod_{i=2}^n e^{\theta B_i}}{Z_n}
= \frac{e^{\theta \haut(\ttD_n)}}{Z_n},
\end{align}
where
\begin{align}
Z_n 
\coloneqq \Ec{e^{\theta \haut(\ttD_n)}}=\Ec{\prod_{i=2}^n e^{\theta B_i}}
= \prod_{i=2}^{n}\left(1+(e^{\theta}-1)\frac{w_i}{W_i}\right).
\end{align}
Then, under this new measure, the random variables $(\widetilde{B}_i)_{2\leq i \leq n}$ and $(J_i)_{1\leq i \leq n}$ still have the same distribution and $(B_i)_{2\leq i \leq n}$ are independent Bernoulli r.v.\@ with respective parameter $p_i$ where
\begin{align}\label{eq:parametres bernoulli apres tilt exponentiel}
p_i\coloneqq \frac{e^{\theta}\frac{w_i}{W_i}}{1+(e^{\theta}-1)\frac{w_i}{W_i}}.
\end{align}

\begin{remark}
In general, we could define $\P_{z}$ in the same way for any other value $z\in \R$ but in this won't be needed for our analysis.
\end{remark} 

\paragraph{Many-to-one}
We first focus on the case of one distinguished point and use Proposition~\ref{prop: construction deux points distingués} for functions $\Psi$ which are defined in such a way that, for any $(\ttt,\ttu)\in \mathbb T_n^{\bullet}$, 
\begin{align*}
\Psi(\ttt,\ttu)= F(\haut(\ttu(1)),\haut(\ttu(2)),\dots ,\haut(\ttu(n))),
\end{align*}
for some function $F:\N^n\rightarrow \R$.
Using Proposition~\ref{prop: construction deux points distingués} and the discussion above, we can write
\begin{align*}
\Ec{\sum_{i=1}^{n}\frac{w_i}{W_n}e^{\theta\haut(\ttu_i)}\Psi(\ttT_n,\ttu_i)}
&= \Ec{e^{\theta\haut(\ttD_n)}\Psi(\ttT_n,\ttD_n)} \\
&= Z_n\cdot \E_{\theta}\left[\Psi(\ttT_n,\ttD_n)\right] \\
&= Z_n\cdot \E_{\theta}\left[F(\haut(\ttD_n(1)),\haut(\ttD_n(2)),\dots ,\haut(\ttD_n(n)))\right]
\end{align*}
Using the description of the sequence $(\haut(\ttD_n(k)))_{1\leq k \leq n}$ from the sequence $(B_2,B_3,\dots ,B_n)$ in \eqref{eq:hauteur point distingué est somme de bernoulli} and the description \eqref{eq:parametres bernoulli apres tilt exponentiel} of the distribution of $(B_2,B_3,\dots ,B_n)$ under $\P_{\theta}$ yields the following statement. 
\begin{lemma}[Many-to-one] \label{lem:many_to_one}
	For any function $F:\N^n\rightarrow \R$ we have
	\begin{align*}
	\Ec{\sum_{i=1}^{n}\frac{w_i}{W_n} e^{\theta\haut(\ttu_i)}
	\cdot F(\haut(\ttu_i(1)),\haut(\ttu_i(2)),\dots ,\haut(\ttu_i(n)))}
	= Z_n\cdot \Ec{F(H_1,H_2,\dots,H_n)},
	\end{align*}
	where $(H_1,H_2,\dots , H_n)$ is such that
	\begin{align*}
		H_k=\sum_{i=2}^{k}\ind{U_i\leq p_i},
	\end{align*}
	for $(U_i)_{i\geq 2}$ i.i.d.\@ uniform random variables on the interval $\intervalleoo{0}{1}$ under $\P$ and $(p_i)_{i\geq 2}$ defined in \eqref{eq:parametres bernoulli apres tilt exponentiel}.
\end{lemma}

\subsection{Many-to-two}

We now apply the same line of reasoning in the case of two distinguished points. 
We fix a function $F:\N^n\rightarrow\R$ and we define a function $\Psi \colon \mathbb T_n^{\bullet}\rightarrow \R$ by 
\begin{equation}
\Psi(\ttt ,\ttu) \coloneqq F(\haut(\ttu(1)),\haut(\ttu(2)),\dots ,\haut(\ttu(n))).
\end{equation}
We also fix a function $f \colon \intervalleentier{1}{n}\rightarrow \R$.
We can use \eqref{eq:égalité en loi deux points distingués} for the function $\Phi:\mathbb T_n^{\bullet \bullet}\rightarrow\R$ such that for every $(\ttt,\ttu,\ttv)\in \mathbb T_n^{\bullet \bullet}$,
\begin{align*}
\Phi(\ttt,\ttu,\ttv)= f(\mathrm{lab}(\ttu\wedge \ttv))\cdot e^{\theta\haut(\ttu)} \cdot \Psi(\ttt,\ttu)\cdot e^{\theta\haut(\ttv)}\cdot \Psi(\ttt,\ttv).
\end{align*} 
This yields
\begin{align}\label{eq:many-to-two premiere etape}
		\Ec{\sum_{1\leq i,j\leq n}\frac{w_i w_j}{W_n^2}\Phi(\ttT_n,\ttu_i,\ttu_j)}
			&=\Ec{f(\mathrm{lab}(\ttD_n\wedge \widetilde{\ttD}_n))\cdot e^{\theta \haut(\ttD_n)} \cdot e^{\theta\haut(\widetilde{\ttD}_n)}\cdot \Psi(\ttT_n,\ttD_n) \cdot \Psi(\ttT_n,\widetilde{\ttD}_n)}\notag\\
		&= \Ec{f(I_n)\cdot e^{\theta \haut(\ttD_n)} \cdot e^{\theta\haut(\widetilde{\ttD}_n)}\cdot \Psi(\ttT_n,\ttD_n) \cdot \Psi(\ttT_n,\widetilde{\ttD}_n)}\notag\\
		&= Z_n\cdot \E_{\theta}\left[f(I_n)\cdot e^{\theta\haut(\widetilde{\ttD}_n)}\cdot \Psi(\ttT_n,\ttD_n) \cdot \Psi(\ttT_n,\widetilde{\ttD}_n)\right]\notag\\
		&= Z_n \cdot \sum_{\ell=1}^n \P_{\theta}(I_n=\ell)\cdot f(\ell) \cdot \E_{\theta} \left[e^{\theta\haut(\widetilde{\ttD}_n)}\cdot \Psi(\ttT_n,\ttD_n) \cdot \Psi(\ttT_n,\widetilde{\ttD}_n) \  \middle| \  I_n=\ell\right],
\end{align}
where we can compute
\begin{align}\label{eq:loi de In}
	\P_{\theta}\left(I_n=\ell\right)=p_\ell q_\ell\cdot \prod_{i=\ell+1}^{n}\left(1-p_iq_i\right),
\end{align}
setting $q_i \coloneqq \frac{w_i}{W_i}$ and recalling the definition of $p_i$ in \eqref{eq:parametres bernoulli apres tilt exponentiel}, with the convention that $p_1=1$.

We can then rewrite the expression appearing in the $\ell$-th term of the sum appearing in \eqref{eq:many-to-two premiere etape} as 
\begin{align*}
&\E_{\theta} \left[e^{\theta\haut(\widetilde{\ttD}_n)}\cdot \Psi(\ttT_n,\ttD_n) \cdot \Psi(\ttT_n,\widetilde{\ttD}_n) \  \middle| \  I_n=\ell\right]\\
&=\E_{\theta} \left[e^{\theta\haut(\widetilde{\ttD}_n)}\cdot F(\haut(\ttD_n(1)),\dots ,\haut(\ttD_n(n))) \cdot F(\haut(\widetilde \ttD_n(1)),\dots ,\haut(\widetilde \ttD_n(n))) \ \middle| \  I_n=\ell\right].
\end{align*}
The random variables in the conditional expectation of the last display only depend on the sequences of Bernoulli random variables $(B_2,\dots, B_n)$ and $(\widetilde{B}_2,\dots, \widetilde{B}_n)$ and so does the conditioning.
By working out explicitly the distribution of $(B_2,\dots B_n,\widetilde{B}_2,\dots, \widetilde{B}_n)$ under $\mathbb P_{\theta}(\ \cdot\  \vert  I_n=\ell)$  we can rewrite the last display as
\begin{align}\label{eq:expression avant de biaiser le deuxieme spine}
\Ec{e^{\theta \widetilde{H}^\ell_n}\cdot F(H^\ell_1,\dots , H^\ell_n)\cdot F(\widetilde{H}^\ell_1,\dots,\widetilde{H}^\ell_n)}
\end{align}
where the sequences $(H^\ell_i)_{1\leq i \leq n}$ and $(\widetilde{H}^\ell_i)_{1\leq i \leq n}$ are defined from two sequences $(U_i)$ and $(V_i)$ of i.i.d.\ uniform random variables on $\intervalleoo{0}{1}$ under $\P$ as follows. For all $1\leq i\leq n$,
\begin{equation}\label{eq:definition H^l}
 H_i^\ell\coloneqq\sum_{j=2}^{i}\ind{U_j\leq p^\ell_j}, \qquad \text{where} \qquad  p^\ell_i\coloneqq \begin{cases}
p_i \quad &\text{for} \quad i<\ell,\\
1 &\text{for}\quad  i=\ell,\\
\frac{p_i (1-q_i)}{1-p_iq_i} &\text{for} \quad i>\ell,\\
\end{cases} 
\end{equation}
 and 
 \begin{equation*}
 		\widetilde{H}^\ell_i= \begin{cases}
 	H^\ell_i & \text{if} \quad i\leq \ell,\\
 	H^\ell_\ell+\sum_{j=\ell+1}^i\ind{V_j\leq \tilde{q}^\ell_j} & \text{if} \quad i> \ell,
 	\end{cases}
 \end{equation*}
 where $\tilde{q}^\ell_i\coloneqq q_i\ind{U_i>p_i^\ell}=q_i\cdot \ind{H^\ell_i=H^\ell_{i-1}}$ for all $\ell+1\leq i \leq n$.
Remark that the $(p^\ell_i)_{2\leq i\leq n}$ are deterministic but the $(\tilde{q}^\ell_i)_{\ell+1\leq i \leq n}$ are random. 
We transform further the expression \eqref{eq:expression avant de biaiser le deuxieme spine}. 
\begin{align*}
	&	\Ec{e^{\theta \widetilde{H}^\ell_n}\cdot F(H^\ell_1,\dots , H^\ell_n)\cdot F(\widetilde{H}^\ell_1,\dots, \widetilde{H}^\ell_n)} \\
	& =\Ec{e^{\theta H^\ell_\ell}\cdot F(H^\ell_1,\dots , H^\ell_n)\cdot \Ecsq{e^{\theta(\widetilde{H}^\ell_n-\widetilde{H}^\ell_\ell)} F(\widetilde{H}^\ell_1,\dots, \widetilde{H}^\ell_n)}{(H^\ell_i)_{1\leq i \leq n}}}
\end{align*}
We can rewrite the conditional expectation using a change of measure as follows:
\[
\Ecsq{e^{\theta(\widetilde{H}^\ell_n-\widetilde{H}^\ell_\ell)} F(\widetilde{H}^\ell_1,\dots, \widetilde{H}^\ell_n)}{(H^\ell_i)_{1\leq i \leq n}}
=\prod_{i=\ell+1}^{n}(1+(e^{\theta}-1)\tilde{q}^\ell_j) \cdot \Ecsq{F(\overline{H}_1^\ell, \dots, \overline{H}^\ell_n)}{(H^\ell_i)_{1\leq i \leq n}},
\]
where $(\overline{H}_i^\ell)_{1\leq i \leq n}$ is defined as 
\begin{equation}\label{eq:definition overline H^l}
\overline{H}^\ell_i= \begin{cases}
H^\ell_i & \text{if} \quad i\leq \ell,\\
H^\ell_\ell+\sum_{j=\ell+1}^i\ind{V_j\leq \tilde{p}^\ell_j}& \text{if} \quad i> \ell,
\end{cases}
\end{equation}
where
\begin{equation*}
\tilde{p}^\ell_i\coloneqq \frac{e^{\theta}\tilde{q}^\ell_i}{1+(e^{\theta}-1)\tilde{q}^\ell_i}=p_i\cdot \ind{H^\ell_i=H^\ell_{i-1}} .
\end{equation*}
In the end, we have
\begin{align*}
&\Ec{e^{\theta \widetilde{H}^\ell_n}\cdot F(H^\ell_1,\dots , H^\ell_n)\cdot F(\widetilde{H}^\ell_1,\dots, \widetilde{H}^\ell_n)} \\
&= \Ec{e^{\theta H^\ell_\ell}\cdot 
\left( \prod_{i=\ell+1}^{n}\left(1+(e^{\theta}-1)q_i \ind{H^\ell_i=H^\ell_{i-1}}\right) \right)
\cdot F(H^\ell_1,\dots , H^\ell_n)\cdot  F(\overline{H}_1^\ell, \dots, \overline{H}^\ell_n)}.
\end{align*}
This yields the following statement.
\begin{lemma}[Many-to-two] \label{lem:many-to-two}
For any $n\geq 1$ and any functions $F \colon \N^n \to \R$ and $f \colon \intervalleentier{1}{n} \to \R$, we have 
\begin{align*}
& \E\Biggl[ \sum_{1\leq i,j\leq n}
\frac{w_iw_j}{W_n^2} f(\mathrm{lab}(\ttu_i\wedge \ttu_j)) 
\cdot e^{\theta\haut(\ttu_i)} F(\haut(\ttu_i(1)),\dots ,\haut(\ttu_i(n)))
\cdot e^{\theta\haut(\ttu_j)} F(\haut(\ttu_j(1)),\dots ,\haut(\ttu_j(n))) \Biggr] \\
& = Z_n \sum_{\ell=1}^n \P_{\theta}(I_n=\ell) \cdot f(\ell)\cdot 
\E\Biggl[ e^{\theta H^\ell_\ell} \prod_{i=\ell+1}^{n}\left(1+(e^{\theta}-1)q_i \ind{H^\ell_i=H^\ell_{i-1}}\right)
\cdot F(H^\ell_1,\dots , H^\ell_n)\cdot  F(\overline{H}_1^\ell, \dots, \overline{H}^\ell_n)
\Biggr],
\end{align*}
where the random sequences $(H_i^\ell)_{1\leq i \leq n}$ and $(\overline{H}_i^\ell)_{1\leq i \leq n}$ appearing above are defined in \eqref{eq:definition H^l} and in \eqref{eq:definition overline H^l}, respectively.
\end{lemma}

\section{Upper bound for the height}
\label{section:upper_bound}

The goal of this section is to prove Theorem~\ref{thm:upper_bound_height_WRT}, which implies in particular the upper bound in Theorem~\ref{thm:height_WRT}. Hence, we work under assumption \eqref{eq:assumption_1}, but not necessarily under assumption \eqref{eq:assumption_2}.

\subsection{Preliminaries}

We first state several consequences of assumption \eqref{eq:assumption_1}, which guarantees the existence of 
$\lambda> 0$ and $\alpha\in\intervalleoo{0}{1}$ such that 
$W_n = \lambda \cdot n^\gamma + O \left( n^{\gamma-\alpha} \right)$.
By \cite[Lemma~3.4]{senizergues2021}, it follows that
\begin{align} \label{eq:consequence_assumption_1}
	\sum_{i=1}^n \frac{w_i}{W_i} 
	= \gamma \log n 
	+ \cst + O(n^{-\alpha})
	\qquad \text{and} \qquad 
	\sum_{i=n}^\infty \left( \frac{w_i}{W_i} \right)^2
	= O(n^{-\alpha}).
\end{align}
Hence, for any $n \geq 0$, we have
\begin{align} \label{eq:bound_Z_n}
	Z_n 
	\leq \exp \left( (e^{\theta}-1) \sum_{i=1}^n \frac{w_i}{W_i} \right)
	\leq C n^{\gamma(e^{\theta}-1)},
\end{align}
where we recall $C$ denotes a positive constant depending only  on $(w_i)_{i\geq 1}$ that can change from line to line.
Moreover, recalling the definition of $p_n$ in \eqref{eq:parametres bernoulli apres tilt exponentiel}, it follows from \eqref{eq:consequence_assumption_1} that 
\begin{align} \label{eq:bound_sum_p_j}
	\sum_{j=1}^n p_j 
	= \gamma e^\theta \log n 
	+ \cst' + O(n^{-\alpha})
	\qquad \text{and} \qquad 
	\sum_{j=n}^\infty p_j^2
	\leq C n^{-\alpha},
\end{align}
as well as the following bound, obtained by subtracting the first part of \eqref{eq:consequence_assumption_1} at $n-1$ and to the one at $n$,
\begin{align} \label{eq:upper_bound_p_n}
	p_n \leq e^\theta \frac{w_n}{W_n} \leq C n^{-\alpha}.
\end{align}
We use these bounds repetitively throughout the section, sometimes without mentioning them.

Recall the random walk $(H_i)_{i\geq 1}$, which appears after applying the many-to-one lemma (see Lemma~\ref{lem:many_to_one}), has Bernoulli$(p_i)$ jumps.
In order to work with an approximately time-homogeneous and centered random walk, 
we introduce, for any $k \geq 0$,
\begin{align*}
	i_k & \coloneqq \inf \left\{ i \geq 1 : \sum_{j=2}^i p_j \geq k \right\},
\end{align*}
and $i_0 \coloneqq 1$. Moreover, we set
\begin{equation}\label{eq:definition Sk}
	S_k \coloneqq H_{i_k} - k.
\end{equation}
This random walk fits the framework of Section~\ref{section:RW} with $\mathbf j=(i_k)_{k\geq 0} $ and  $\mathbf r = (p_j)_{j\geq 2}$.

We need the following estimates for this time change.
Note that \eqref{eq:bound_sum_p_j} implies that 
\begin{align} \label{eq:bound_i_k}
	k - C
	\leq \gamma e^{\theta} \log i_k
	\leq k + C.
\end{align}
For any $n \geq 1$, let $\tau(n)$ denote the smallest integer $t$ such that $i_t \geq n$. It follows from \eqref{eq:bound_i_k} that
\begin{align} \label{eq:bound_tau(n)}
	\gamma e^{\theta} \log n - C
	\leq \tau(n) 
	\leq \gamma e^{\theta} \log n + C.
\end{align}
Moreover, in our case, the quantities introduced in \eqref{eq:def_delta}, which appear in the estimates for the random walk $(S_k)_{k\geq 0}$, can be bounded as follows
\begin{equation} \label{eq:bounds_delta_eta} 
	\delta_k \leq p_{i_k} 
	\leq C e^{-ck},
	\qquad
	\Delta_k \leq 1
	\qquad \text{and} \qquad
	\eta_k \leq C \exp\left(-c k^c\right),
\end{equation}
using \eqref{eq:upper_bound_p_n}, \eqref{eq:bound_i_k} and \eqref{eq:bound_sum_p_j}.
\begin{remark}
We repetitively need upper bounds for quantities of the form
\begin{align} \label{eq:quantity}
	& \Ec{e^{\theta\haut(\ttu_n)} 
	F(\haut(\ttu_n(i_0)),\haut(\ttu_n(i_1)),\dots,\haut(\ttu_n(i_{t-1})))
	g(\haut(\ttu_n)) },
\end{align}
with $n \geq 1$, $t \coloneqq \tau(n)$ and $F \colon \R^{t} \to \R_+$ and $g \colon \R \to \R_+$ measurable functions.
In order to avoid the repetition of the same argument, we explain here how we proceed.
Using the dynamics of the construction, conditionally on the tree $\ttT_{n-1}$, the vertex $\mathtt{u}_n$ is the child of any vertex $\mathtt{u}_j$ with $1\leq j\leq n-1$ with probability $\frac{w_j}{W_{n-1}}$. Note that in that case we have $\haut(\mathtt{u}_n)=\haut(\mathtt{u}_j(n-1))+1$.
Taking the conditional expectation with respect to $\ttT_{n-1}$, we get that \eqref{eq:quantity} equals
\begin{align*}
	& \Ec{\sum_{j=1}^{n-1} \frac{w_j}{W_{n-1}}
	e^{\theta(\haut(\ttu_j(n-1))+1)} 
	F(\haut(\ttu_j(i_0)),\dots,\haut(\ttu_j(i_{t-1})))
	g(\haut(\ttu_j(n-1))+1) } \\
	& \leq \frac{W_{i_t}}{W_{n-1}} 
	\Ec{\sum_{j=1}^{i_t} \frac{w_j}{W_{i_t}}
	e^{\theta(\haut(\ttu_j(i_t))+1)} 
	F(\haut(\ttu_j(i_0)),\dots,\haut(\ttu_j(i_{t-1})))
	g(\haut(\ttu_j(i_t))+1) },
\end{align*}
using $i_t \geq n \geq n-1$ to add non-negative terms in the sum and noting that, for $j \leq n-1$, $\ttu_j(n-1) = \ttu_j = \ttu_j(i_t)$.
Applying the many-to-one lemma (Lemma~\ref{lem:many_to_one}), the right-hand side of the last displayed equation equals
\begin{align*}
	\frac{W_{i_t}}{W_{n-1}} Z_{i_t} e^{\theta}
	\Ec{F(H_{i_0},\dots,H_{i_{t-1}})
	g(H_{i_t}+1) },
\end{align*}
By \eqref{eq:bound_i_k} and \eqref{eq:bound_tau(n)}, note that $i_t \leq C n$ for any $n \geq 2$ and it follows from \eqref{eq:assumption_1} that $W_{i_t}/W_{n-1} \leq C$.
Combining the above and using \eqref{eq:bound_Z_n}, we get
\begin{equation} \label{eq:one_to_one}
	\Ec{e^{\theta\haut(\ttu_n)} 
	F(\haut(\ttu_n(i_0)),\dots,\haut(\ttu_n(i_{t-1}))) 
	g(\haut(\ttu_n)) }
	\leq C n^{\gamma(e^{\theta}-1)}
	\Ec{F(H_{i_0},\dots,H_{i_{t-1}})
	g(H_{i_t}+1) }.
\end{equation}
\end{remark}

\subsection{Introducing the first barrier}

\begin{lemma} \label{lem:first_barrier}
Recall $\tau(n) \coloneqq \min\{ t \in \N : i_t \geq n\}$. There exists $C > 0$, such that, for any integer $K \geq 0$,
\[
	\P \left( \exists n \geq 1 : \haut(\ttu_n) > \tau(n) + K \right) 
	\leq C (K+1) e^{-\theta K}.
\]
\end{lemma}
Note that by \eqref{eq:bound_tau(n)}, a similar statement could be made with $\gamma e^{\theta} \log n$ instead of $\tau(n)$. 
However, this formulation is more convenient to prove and fits exactly our future purpose.
\begin{proof}
Let $B \coloneqq \{ \exists n \geq 1 : \haut(\ttu_n) > \tau(n) + K  \}$ denote the event we want to control.
Distinguishing according to the first integer $n$ such that $\haut(\ttu_n) > \tau(n) + K$, we have $B = \bigcup_{n\geq2} B_n$ where we set
\[
	B_n \coloneqq \{ \haut(\ttu_n) > \tau(n) + K,\ \forall m < n,
	 \haut(\ttu_m) \leq \tau(m) + K \}.
\]
On the event $B_n$, we have $\haut(\ttu_n(n-1)) \leq \tau(n-1) + K \leq \tau(n) + K$. 
But, on the other hand, note that $\haut(\ttu_n(n-1)) = \haut(\ttu_n) - 1$ so we necessarily have $\haut(\ttu_n) = \tau(n) + K +1$.
Hence, keeping only part of the constraints, we have
\begin{align*}
	\P(B_n)
	& \leq \Pp{\haut(\ttu_n) = \tau(n) + K + 1, \forall k < \tau(n), \haut(\ttu_n(i_k)) \leq k + K} \\
	& = e^{-\theta(\tau(n)+K + 1)} \cdot 
	\Ec{e^{\theta \haut(\ttu_n)} \1_{\{\haut(\ttu_n) = \tau(n) + K + 1\}} 
	\1_{\{\forall k < \tau(n),\ \haut(\ttu_n(i_k)) \leq k + K\}}} \\
	& \leq C e^{-\theta(\tau(n)+K)} n^{\gamma(e^{\theta}-1)}
	\Pp{H_{i_{\tau(n)}} +1 = \tau(n) + K + 1,\
		\forall k < \tau(n), H_{i_k} \leq k + K},
\end{align*}
applying \eqref{eq:one_to_one}. Recalling the definition of the walk $(S_k)$ in \eqref{eq:definition Sk}, this last probability equals
\begin{align*}
	\Pp{S_{\tau(n)} =  K,\ \forall k < \tau(n), S_k \leq K}
	& \leq \frac{C (K+1)}{\tau(n)^{3/2}},
\end{align*}
by Lemma~\ref{lem:upper_bound_RW_double_barrier} and \eqref{eq:bounds_delta_eta}.
Therefore, using \eqref{eq:bound_tau(n)} and $\gamma (e^{\theta}- 1 -\theta e^{\theta})=-1$, we finally obtain
\begin{align*}
	\P(B_n)
	& \leq \frac{C (K+1) e^{-\theta K}}{n (\log n)^{3/2}},
\end{align*}
and the result follows by summing over $n \geq 2$.
\end{proof}

\subsection{Proof of the upper bound for the height}

We now state and prove a key lemma for the proof of Theorem~\ref{thm:upper_bound_height_WRT}.
Let
\begin{align*}
	x_n & \coloneqq \left\lfloor \frac{3}{2\theta} \log \log n \right\rfloor.
\end{align*}
\begin{lemma} \label{lem:upper_bound_WRT}
	There exist constants $C,c>0$ such that for any integers $L,K \geq 0$, $a \geq K$ and $n \geq 1$, we have, setting $t \coloneqq \tau(n)$,
	\begin{align*}
		& \P \Bigg( \exists \ttu \in \ttT_n : 
			\haut(\ttu) = t-x_n+a,\ 
			\max_{k \in \intervalleentier{0}{t}} \haut(\ttu(i_k)) - k \leq K,\
			\max_{k \in \intervalleentier{t/2}{t}} \haut(\ttu(i_k)) - k = -x_n + a + L\Bigg) \\
		& \leq C (K+1) e^{-\theta L/4} e^{-\theta a}.
	\end{align*}
\end{lemma}
The proof of this lemma is very close to the proof of Lemma~3.3 of Aïdékon \cite{aidekon2013} for the branching random walk, up to additional technicalities due to our model.
\begin{proof}
For brevity, we introduce $S_k(\ttu) \coloneqq \haut(\ttu(i_k)) - k$, which is exactly what is transformed into $S_k$ after applying the many-to-one lemma.
Let $E$ denote the event we are interested in.
We first distinguish according to the instant $j \in \intervalleentier{t/2}{t}$ where $\max_{k \in \intervalleentier{t/2}{t}} S_k(\ttu)$ is reached: we introduce, for any $1\leq m\leq n$, the event
\[
	E_j(m) 
	\coloneqq \Bigl\{ 
		S_t(\ttu_m) = -x_n+a,\ 
		\max_{k \in \intervalleentier{0}{t}} S_k(\ttu_m) \leq K,\ 
		\max_{k \in \intervalleentier{t/2}{t}} S_k(\ttu_m) = S_j(\ttu_m) = -x_n + a + L
	\Bigr\}
\]
and then $E_j \coloneqq \bigcup_{1\leq m \leq n} E_j(m)$.
Note here that $i_t \geq n \geq m$ so $\ttu_m(i_t) = \ttu_m$  and therefore $S_t(\ttu_m) = \haut(\ttu_m) - t$.
By the union bound, we have $\Pp{E} \leq \sum_{j \in \intervalleentier{t/2}{t}} \Pp{E_j}$, so we now have to bound $\Pp{E_j}$.
For this, we distinguish the cases $t/2 \leq j \leq t-2b$ and $t-2b < j \leq t$ with
\[
	b \coloneqq \left\lceil \frac{2}{\theta} e^{\theta L/2} \right\rceil,
\]
which satisfies $b > L$ for any $L\geq0$.
Moreover, we can restrict ourselves to the case where $L \leq K-a+x_n$, otherwise the probability in the lemma is simply zero.
This implies that $L \leq x_n$ and therefore $b \leq \frac{2}{\theta} (\log n)^{3/4}+1$. 
Hence, we consider from now $n$ large enough (independently of $K,a,L$) such that $b \leq t/8$. 
The case where $n$ is small is immediate by choosing the constant $C$ in the lemma large enough.

Start with the case $t/2 \leq j \leq t-2b$. We write, recalling that $S_t(\ttu_m) = \haut(\ttu_m) - t$,
\begin{align} \label{eq:first_bound_P(E_j)}
	\Pp{E_j}
	& \leq \Ec{ \sum_{m=1}^n \1_{E_j(m)}}
	= e^{-\theta(t - x_n+a)} 
	\sum_{m=1}^n \Ec{ e^{\theta \haut(\ttu_m)} \1_{E_j(m)}}.
\end{align}
We fix some $m \in \llbracket 1,n \rrbracket$ and let $s = \tau(m)$ be the smallest integer such that $i_s \geq m$.
If $s \leq t-L-1$, we have $\ttu_m(i_{t-L-1}) = \ttu_m$ and therefore, on the event $E_j(m)$,
\[
	S_{t-L-1}(\ttu_m) 
	= \haut(\ttu_m)-t+L+1
	= S_t(\ttu_m)+L+1
	= - x_n + a + L+1,
\]
which is a contradiction because $t-L-1 \geq t-b-1 \geq t/2$ and on that event we have $\max_{k \in \intervalleentier{t/2}{t}} S_k(\ttu_m)\leq - x_n + a + L$.
Hence, the event $E_j(m)$ is empty for any $m \leq i_{t-L-1}$ and we can restrict ourselves to the case $m \in \llbracket i_{t-L-1} + 1,n \rrbracket$.
Then, we have $j \leq t-2b \leq t-L-1 \leq s-1$, so $E_j(m)$ is contained in the event
\[
	  \biggl\{ 
		\haut(\ttu_m)-s = - x_n+a+t-s,\ 
		\max_{k \in \intervalleentier{0}{s-1}} S_k(\ttu_m) \leq K,\
		\max_{k \in \intervalleentier{t/2}{s-1}} S_k(\ttu_m) = S_j(\ttu_m) = - x_n + a + L
	\biggr\}.
\]
Applying \eqref{eq:one_to_one} with $m$ and $s$ instead of $n$ and $t$, we get
\begin{align*}
	& \Ec{ e^{\theta \haut(\ttu_m)} \1_{E_j(m)}} \\
	& \leq C m^{\gamma(e^{\theta}-1)}
	\P \biggl( S_s+1 = -x_n+a+t-s,\ 
		\max_{k \in \intervalleentier{0}{s-1}} S_k \leq K,\
		\max_{k \in \intervalleentier{t/2}{s-1}} S_k = S_j = -x_n + a + L \biggr) \\
	& \leq C m^{\gamma(e^{\theta}-1)}
	\Pp{\max_{k \in \intervalleentier{0}{j}} S_k \leq K,\ 
		\max_{k \in \intervalleentier{t/2}{j}} S_k = S_j = -x_n + a + L } \\
	& \hspace{3.64cm} {} \cdot 
	\Pp{\overline{S}_{s-j} = t-s-1-L,\
		\max_{k \in \intervalleentier{0}{s-j}} \overline{S}_k \leq 0},
\end{align*}
applying the Markov property at time $j$, setting $\overline{S}_k \coloneqq S_{j+k} - S_j$ and using that $s \geq t-L$ to extend $\max_{k \in \intervalleentier{0}{s-j-1}} \overline{S}_k$ to the time $k = s-j$.
Note that the random walk $\overline{S}$ fits also the framework of Section~\ref{section:RW}
and that the quantities in \eqref{eq:def_delta} are bounded as follows in that case: $\Delta_k \leq 1$ and $\eta_k \leq C e^{ct^c}$ for any $k \geq 0$, using that $j \geq t/2$.
Applying Lemma~\ref{lem:upper_bound_RW_double_barrier}, we get
\begin{align*}
	\Pp{\overline{S}_{s-j} = t-s-1-L,
		\max_{k \in \intervalleentier{0}{s-j}} \overline{S}_k \leq 0}
	& \leq \frac{C (L+2+s-t)}{(s-j)^{3/2}}
	\leq \frac{C (L+1)}{(t-b-j)^{\frac32}},
\end{align*}
where the last inequality comes from the fact that $s\geq t-L>t-b> j$. 
On the other hand, applying Lemma~\ref{lem:upper_bound_RW_double_barrier}  and \eqref{eq:bounds_delta_eta}, we have
\begin{align*}
	\Pp{\max_{k \in \intervalleentier{0}{j}} S_k \leq K, 
		\max_{k \in \intervalleentier{t/2}{j}} S_k = S_j = -x_n + a + L }
	& \leq 
	\begin{cases}
	\frac{C(K+1)}{t^{3/2}} \vphantom{\frac{p}{\frac{p}{p}}} 
	& \text{if } \frac{3t}{4} \leq j \leq t, \\
	\frac{C(K+1) (x_n+1)}{t^{3/2}} 
	& \text{if } \frac{t}{2} \leq j < \frac{3t}{4},
	\end{cases}
\end{align*}
where in the second case we simply omit the constraint $\max_{k \in \intervalleentier{t/2}{j}} S_k \leq -x_n + a + L$ and use that $K+x_n-a-L \leq x_n$.
Coming back to \eqref{eq:first_bound_P(E_j)} and using \eqref{eq:bound_tau(n)}, we proved
\begin{align*}
	\Pp{E_j}
	& \leq C e^{-\theta a} (\log n)^{3/2} n^{-\theta \gamma e^{\theta}} \sum_{m=i_{t-L-1}+1}^{n-1} 
	\Ec{ e^{\theta \haut(\ttu_m)} \1_{E_j(m)}} \\
	& \leq C e^{-\theta a} (\log n)^{3/2} 
	\cdot \frac{(K+1)}{t^{3/2}} 
	\left(1 + x_n \1_{t/2 \leq j < 3t/4} \right)
	\cdot \frac{(L+1)}{(t-b-j)^{3/2}},
\end{align*}
bounding the number of terms in the sum by $n$ and using that $\gamma(e^{\theta}-1-\theta e^{\theta}) = -1$.
Hence, we get
\begin{align*}
	\sum_{j=t/2}^{t-2b} \Pp{E_j}
	& \leq  C e^{-\theta a} (K+1) (L+1)
	\left( 
	\sum_{j=t/2}^{3t/4-1}
	\frac{(1+x_n)}{(t-b-j)^{3/2}}
	+
	\sum_{j=3t/4}^{t-2b}
	\frac{1}{(t-b-j)^{3/2}}
	\right) \\
	& \leq  C e^{-\theta a} (K+1) (L+1)
	\left( 
	\frac{\log \log n}{(\log n)^{1/2}}
	+
	\frac{1}{b^{1/2}} \right).
\end{align*}
Since $b \leq \frac{2}{\theta} (\log n)^{3/4}+1$, we have $\log \log n/(\log n)^{1/2} \leq C/b^{1/2}$.
Then, recalling the definition of~$b$, this gives the desired bound for this part of the sum over~$j$.

We now deal with the case $t-2b < j \leq t$. 
Note that, forgetting the constraints on $S_k(\ttu)$ for $k > j$, 
\begin{align*}
	E_j 
	& \subset \Bigl\{ \exists u \in \ttT_{i_j} : 
		\max_{k \in \intervalleentier{0}{j}} S_k(\ttu) \leq K,\ 
		\max_{k \in \intervalleentier{t/2}{j}} S_k(\ttu) = S_j(\ttu) = -x_n+a+L
	\Bigr\}.
\end{align*}
Then, proceeding similarly as before, 
\[
	\Pp{E_j}
	\leq e^{-\theta(j -x_n+a+L)} \sum_{m=i_{j-1}+1}^{i_j} 
	\E\Bigl[ e^{\theta \haut(\ttu_m)} 
	\1_{\lbrace\max_{k \in \intervalleentier{0}{j}} S_k(\ttu_m) \leq K\rbrace}
	\1_{\lbrace\max_{k \in \intervalleentier{t/2}{j}} S_k(\ttu_m) = S_j(\ttu_m) = -x_n+a+L\rbrace} \Bigr],
\]
where we noted that the event in the indicator function is empty if $m \leq i_{j-1}$ because in that case $S_{j-1}(\ttu_m) = S_{j}(\ttu_m) +1$.
In particular, note that here $\tau(m)=j$.
Using \eqref{eq:one_to_one} as before, the last expectation is smaller than
\begin{align*}
	&C m^{\gamma(e^{\theta}-1)}
	\Pp{\max_{k \in \intervalleentier{0}{j-1}} S_k \leq K,\  
		\max_{k \in \intervalleentier{t/2}{j-1}} S_k \leq -x_n+a+L,\
		S_j+1 = -x_n+a+L}\\
	&\leq C m^{\gamma(e^{\theta}-1)} \frac{(K+1)}{t^{3/2}},
\end{align*}
by Lemma~\ref{lem:upper_bound_RW_double_barrier}, noting that $j \geq t-2b \geq 3t/4$.
Hence, we get
\begin{align*}
	\sum_{j=t-2b+1}^{t} \Pp{E_j}
	& \leq C e^{-\theta (a+L)} (K+1) 
	\sum_{j=t-2b+1}^{t} \sum_{m=i_{j-1}+1}^{i_j} 
	(i_j)^{\gamma(e^{\theta}-1)} e^{-\theta j}
	\leq C e^{-\theta (a+L)} (K+1) b
\end{align*}
using that $e^{-\theta j} \leq C (i_j)^{-\theta \gamma e^{\theta}}$, bounding the number of terms in the sum over $m$ by $i_j$ and using again that $\gamma(e^{\theta}-1-\theta e^{\theta}) = -1$.
This concludes the proof.
\end{proof}
We can now proceed to the proof of Theorem~\ref{thm:upper_bound_height_WRT}, which implies the upper bound in Theorem~\ref{thm:height_WRT}.
\begin{proof}[Proof of Theorem~\ref{thm:upper_bound_height_WRT}]
Setting $t \coloneqq \tau(n)$ and using \eqref{eq:bound_tau(n)}, it is enough to prove, for any $b \geq 1$,
\begin{align*}
	\Pp{\exists \ttu \in \ttT_n : \haut(\ttu) \geq t-x_n+b} 
	& \leq C b e^{-\theta b}.
\end{align*}
For this, we first apply Lemma~\ref{lem:first_barrier} with $K = b$ to work on the event $\{\forall m \geq 1,\ \haut(\ttu_m) \leq \tau(m) + b \}$.
Then, we apply Lemma~\ref{lem:upper_bound_WRT} with all possible values of $a\geq b$ and $L\geq 0$, and $K=b$. The result follows from a union-bound.
%
%
\end{proof}

\section{Lower bound for the height}
\label{section:lower_bound}

\subsection{Strategy}

For any integer $N\geq 1$, we construct a new tree $\ttT_n^{(N)}$ from $\ttT_n$: we first remove all vertices with labels 2 through $N$ and then attach all of them and all of their children to the root.
Note that $\ttT_n^{(N)}$ has distribution $\wrt(\boldsymbol{w}^{(N)})$, 
where the sequence of weights $\boldsymbol w^{(N)}$ is related to the sequence $\boldsymbol w$ as follows:
\begin{align} \label{eq:def_w^(N)}
	w_i^{(N)}=
	\begin{cases}
	W_N & \text{if} \quad i=1,\\
	0 & \text{if} \quad 2\leq i\leq N,\\
	w_i & \text{if} \quad i\geq N+1.
	\end{cases}
\end{align}
In other words, the sequence $\boldsymbol w^{(N)}$ is obtained from $\boldsymbol w$ by "transferring" all the weight of vertices $2$ to $N$ to the first vertex, and leaving the rest unchanged. 
Our aim is to prove a lower bound for the height of $\ttT_n^{(N)}$, and the lower bound for $\ttT_n$ follows because $\haut(\ttT_n) \geq \haut(\ttT_n^{(N)})$.

We introduce the following quantities associated to this new sequence of weights $\boldsymbol w^{(N)}$:
\[
	W_n^{(N)} \coloneqq \sum_{j=2}^n w_j^{(N)}
	\qquad 
	p_n^{(N)} \coloneqq 
	\frac{e^{\theta}\frac{w_n^{(N)}}{W_n^{(N)}}}
		{1+(e^{\theta}-1)\frac{w_n^{(N)}}{W_n^{(N)}}},
	\qquad 
	i_k^{(N)} \coloneqq \inf \left\{ i \geq 1 : \sum_{j=2}^i p_j^{(N)} \geq k \right\}.
\]
As before we define
\[
x_n \coloneqq \left\lfloor \frac{3}{2\theta} \log \log n \right\rfloor.
\]
Then, for some $K \geq 0$, the quantity we use for our first and second moment argument is the following: for $n=i_t^{(N)}$, 
\begin{equation} \label{def:Q_n^(N)}
	Q_n^{(N)} 
	\coloneqq \sum_{m=1}^{n} \frac{w_m^{(N)}}{W_n^{(N)}} 
	e^{\theta \haut(\ttu_m)}
	\ind{\haut(\ttu_m) = t - x_n} 
	\ind{\max_{k \in \intervalleentier{0}{t/2}} \haut(\ttu_m(i_k^{(N)})) - k \leq K} 
	\ind{\max_{k \in \intervalleentier{t/2}{t}} \haut(\ttu_m(i_k^{(N)})) - k \leq -x_n},
\end{equation}
where $\haut(\ttu_m)$ refers implicitly to the height of $\ttu_m$ in $\ttT_n^{(N)}$. Note that the dependence of $Q_n^{(N)}$ in $K$ is also kept implicit.
The following lemma gives bounds for the first and second moment of $Q_n^{(N)}$ and is proved in Sections \ref{subsection:first_moment_Q} and \ref{subsection:second_moment_Q}.
\begin{lemma} \label{lem:first_and_second_moment}
For any $\varepsilon > 0$, there exist $K_0(\varepsilon),N_0(\varepsilon),n_0(\varepsilon)$ such that for any $K \geq K_0(\varepsilon)$, $N \geq N_0(\varepsilon)$ and $n \geq n_0(\varepsilon)$ such that $n = i_t^{(N)}$ for some $t \geq 1$,
we have
\begin{align}
	\Ec{Q_n^{(N)}} 
	& \geq \frac{Z_n^{(N)}}{t^{3/2}} \cdot (1-\varepsilon) \sqrt{\frac{2}{\pi}} 	
	\frac{K}{\rho^-},
	\label{eq:first_moment_Q_n} \\
	\Ec{\left( Q_n^{(N)} \right)^2} 
	& \leq \left( \frac{Z_n^{(N)}}{t^{3/2}} \cdot (1+\varepsilon) \sqrt{\frac{2}{\pi}} 	
		\frac{K}{\rho^-} \right)^2,
	\label{eq:second_moment_Q_n}
\end{align}
where the constant $\rho^-$ is defined in Section~\ref{subsection:known_results_RW}.
\end{lemma}
\begin{proof}[Proof of the lower bound in Theorem~\ref{thm:height_WRT}] 
Consider some $b \geq 0$.
The tree $\ttT_n$ is higher than $\ttT_n^{(N)}$ so
\begin{align}
	\Pp{\haut(\ttT_n) \geq \gamma e^{\theta} \log n - x_n - b} 
	& \geq \Pp{\haut(\ttT_n^{(N)}) \geq \gamma e^{\theta} \log n - x_n - b} \nonumber \\
	& \geq \Pp{\haut(\ttT_m^{(N)}) \geq t + \gamma e^{\theta} \log N + C_0 - x_m - b},
\end{align}
with $t$ such that $i_t^{(N)} \leq n < i_{t+1}^{(N)}$ and $m \coloneqq i_t^{(N)}$ and using that $t \geq \gamma e^\theta (\log n - \log N)-C_0$ by \eqref{eq:estimate_i_k^(N)}, where $C_0$ is a constant.
Now fix some $\varepsilon$.
We take $K=K_0(\varepsilon)$ and $N = N_0(\varepsilon)$ given by Lemma~\ref{lem:first_and_second_moment} and assume that $n$ is large enough such that $m \geq n_0(\varepsilon)$.
Then, with $b = \gamma e^{\theta} \log N + C_0$, we get
\begin{align*}
	\Pp{\haut(\ttT_n) \geq \gamma e^{\theta} \log n - x_n - b} 
	\geq \Pp{\haut(\ttT_m^{(N)}) \geq t - x_m}
	\geq \Pp{Q_m^{(N)} > 0}
	\geq \frac{\E[Q_m^{(N)}]^2}{\E[(Q_m^{(N)})^2]},
\end{align*}
by Cauchy--Schwarz inequality.
By Lemma~\ref{lem:first_and_second_moment}, for any $\varepsilon > 0$, there exists $b \in \R$ such that
\begin{align*}
	\Pp{\haut(\ttT_n) \geq \gamma e^{\theta} \log n - x_n - b} 
	\geq \frac{(1-\varepsilon)^2}{(1+\varepsilon)^2},
\end{align*}
which proves the lower bound in Theorem~\ref{thm:height_WRT}.
\end{proof}

\subsection{Preliminaries}

Recall we work with an initial sequence $\boldsymbol w$ that satisfies assumption \eqref{eq:assumption_1} for some $\gamma>0$ and \eqref{eq:assumption_2}.
In this section, we list some bounds for the quantities depending on the modified sequence $\boldsymbol w^{(N)}$.
Anytime we add a superscript $(N)$ to a symbol that was implicitly a function of the weight sequence $\boldsymbol{w}$, it corresponds to the analog object for the weight sequence $\boldsymbol w^{(N)}$. 
Constants $C,c>0$ that can change from line to line and $O(\dots)$ terms can only depend on the initial sequence $\boldsymbol w$, but \textbf{not} on $N$.
Moreover, we denote by $\kappa_N$ a quantity that depends only on $N$,  tends to 0 as $N \to \infty$ and can change from line to line.

It follows from \eqref{eq:consequence_assumption_1} that, for $n \geq N+1$,
\begin{align} \label{eq:sum_w_i^(N)}
	\sum_{i=2}^{n}\frac{w_i^{(N)}}{W_i^{(N)}}
	& = \sum_{i=N+1}^{n}\frac{w_i}{W_i}
	=\gamma (\log n - \log N) + \kappa_N + O(n^{-\alpha}), \\
	\label{eq:sum_p_i^(N)}
	\sum_{i=2}^{n}p^{(N)}_i
	& = \sum_{i=N+1}^{n}p_i
	= \gamma e^{\theta} (\log n - \log N) + \kappa_N + O(n^{-\alpha}), \\
	\label{eq:bound_Z_n^(N)}
	Z_n^{(N)}
	& = \prod_{i=N+1}^{n}\left(1+(e^{\theta}-1)\frac{w_i}{W_i}\right)
	= \left( \frac{n}{N} \right)^{\gamma(e^{\theta}-1)}
		\exp \left( \kappa_N + O(n^{-\alpha}) \right).
\end{align}
Moreover, we have, for any $n \geq 1$,
\begin{align} \label{eq:bound_p_i^(N)}
	p_n^{(N)}\leq C n^{-\alpha},
\end{align}
and, by assumption \eqref{eq:assumption_2},
\begin{align} \label{eq:sum_p_i^(N)_square}
	\sum_{i=n}^\infty \left( p^{(N)}_i \right)^2 
	= \sum_{i=n \vee N}^\infty \left( p_i \right)^2 
	\leq  \frac{C}{n \vee N}. 
\end{align}
Concerning the time change $i_k^{(N)}$, one can check that, for any $k \geq 1$, 
\begin{align} \label{eq:estimate_i_k^(N)}
	i_k^{(N)}
	= N\cdot \exp\left(\frac{k}{\gamma e^{\theta}} + \kappa_N + O \left(e^{-ck}\right) \right),
\end{align}
using \eqref{eq:bound_p_i^(N)} and the fact that $i_k^{(N)} \geq i_1^{(N)} \geq N+1$. 
Note that the last display hold for $i_k^{(N)}\vee N$ for all $k\geq 0$.


The remaining part of this section is dedicated to the proof of Lemma~\ref{lem:first_and_second_moment}. 
From now on, we consider the tree associated to the sequence $\boldsymbol w^{(N)}$, but we \textbf{omit the dependence on} $N$ in notation, writing for example $w_i$ instead of $w_i^{(N)}$.

\subsection{First moment}
\label{subsection:first_moment_Q}

In this section, we prove \eqref{eq:first_moment_Q_n}.
Applying the many-to-one lemma and setting $S_k \coloneqq H_{i_k} - k$, we have
\begin{align*}
	\Ec{Q_n} 
	& = Z_n \cdot
	\Pp{S_t = - x_n,\ 
		\max_{k \in \intervalleentier{0}{t/2}} S_k \leq K,\
		\max_{k \in \intervalleentier{t/2}{t}} S_k \leq -x_n} 
	\geq Z_n \cdot (1-\varepsilon) \sqrt{\frac{2}{\pi}} 
		(K-1) \frac{1}{\rho^- t^{3/2}},
\end{align*}
applying Lemma~\ref{lem:estimate_RW_double_barrier} and noting that $R^-(0) = 1$.
The result follows.

\subsection{Second moment}
\label{subsection:second_moment_Q}

In this section, we prove \eqref{eq:second_moment_Q_n}, up to Lemma~\ref{lem:estimates_RW_for_second_moment}, which we prove in Section \ref{subsection:applying_the_RW_estimates}. Recall we assumed $n=i_t$.
We apply the many-to-two lemma (Lemma~\ref{lem:many-to-two}) with
\[
	F(\haut(\ttu(1)),\dots ,\haut(\ttu(n)))
	\coloneqq 
	\ind{\haut(\ttu(n)) = t - x_n} 
	\ind{\max_{k \in \intervalleentier{0}{t/2}} \haut(\ttu(i_k)) - k \leq K}
	\ind{\max_{k \in \intervalleentier{t/2}{t}} \haut(\ttu(i_k)) - k \leq -x_n},
\]
in order to get 
\begin{align} \label{eq:first_bound_second_moment_Q_n}
	\Ec{ Q_n^2 }
	& \leq Z_n^2 \cdot 
	\sum_{\ell=1}^n \frac{\P_{\theta}(I_n=\ell)}{Z_\ell} 
	\cdot 
	\Ec{e^{\theta H^\ell_\ell}\cdot F(H^\ell_1,\dots , H^\ell_n) \cdot  F(\overline{H}_1^\ell, \dots, \overline{H}^\ell_n)},
\end{align}
where we bounded $\prod_{i=\ell+1}^{n} (1+(e^{\theta}-1)q_i \ind{H^\ell_i=H^\ell_{i-1}})$ by $\prod_{i=\ell+1}^{n} (1+(e^{\theta}-1)q_i)= Z_n/Z_\ell$.
The following lemma gives us bounds for the expectation on the right-hand side of \eqref{eq:first_bound_second_moment_Q_n}. We postpone its proof to the next section.
\begin{lemma} \label{lem:estimates_RW_for_second_moment}
Let $1 \leq \ell \leq n$ and let $s$ be the smallest integer such that $i_s \geq \ell$. Let $K \geq 0$ and $N \geq K^{2}$.
\begin{enumerate}
\item\label{lem:estimates_RW_for_second_moment:it:large s} If $s \geq 3t/4$, then
\[
	\Ec{e^{\theta H^\ell_\ell}\cdot F(H^\ell_1,\dots , H^\ell_n) \cdot  F(\overline{H}_1^\ell, \dots, \overline{H}^\ell_n)}
	\leq 
	\frac{C e^{\theta (s-x_n)}}{(t-s)^3+1} \frac{(K +1)}{t^{3/2}}.
\]
\item\label{lem:estimates_RW_for_second_moment:it:intermediate s} If $1 \leq s < 3t/4$, then
\[
	\Ec{e^{\theta H^\ell_\ell}\cdot F(H^\ell_1,\dots , H^\ell_n) \cdot  F(\overline{H}_1^\ell, \dots, \overline{H}^\ell_n)}
	\leq 
	\frac{C e^{\theta (s+K)}}{t^3} \frac{(K +1)}{s^{3/2}}.
\]
\item\label{lem:estimates_RW_for_second_moment:it:first term} If $\ell = 1$, we consider some $\varepsilon > 0$ and let $K_0$ and $n_0$ be the constants given by Lemma~\ref{lem:estimate_RW_double_barrier}.
If $n \geq n_0$ and $K \in \llbracket K_0, n^{1/4} \rrbracket$, then
\[
	\Ec{e^{\theta H^\ell_\ell}\cdot F(H^\ell_1,\dots , H^\ell_n) \cdot  F(\overline{H}_1^\ell, \dots, \overline{H}^\ell_n)} 
	\leq 
	\left( (1+\varepsilon) \sqrt{\frac{2}{\pi}} 
			\frac{(K+ C N^{-c})}{\rho^- t^{3/2}} \right)^2.
\]
\end{enumerate}
\end{lemma}
We now apply this lemma to conclude the proof of \eqref{eq:second_moment_Q_n}. 
We break the sum on the right-hand side of \eqref{eq:first_bound_second_moment_Q_n} into three terms $T_1 + T_2 + T_3$, where $T_1$ corresponds to the part where $\ell = 1$, $T_2$ to the part $2 \leq \ell \leq i_{\lceil 3t/4 \rceil - 1}$ and $T_3$ to the part $i_{\lceil 3t/4 \rceil - 1} < \ell \leq n$.
First note that, for any $k \geq 1$,
\begin{align*} 
	\P_{\theta}(I_n=k)
	= p_k q_k\cdot \prod_{i=k+1}^{n}\left(1-p_iq_i\right)
	\leq p_k q_k \leq p_k^2,
\end{align*}
recalling that $q_k = w_k/W_k \leq p_k$.
Start with $T_1$, which is the main term.
Since $Z_1 = 1$ and $p_1=1$, we get by Lemma~\ref{lem:estimates_RW_for_second_moment}\ref{lem:estimates_RW_for_second_moment:it:first term}
\begin{align} \label{eq:bound_T_1}
	T_1 
	\leq \left( (1+\varepsilon) \sqrt{\frac{2}{\pi}} 
		\frac{(K+ C N^{-c})}{\rho^- t^{3/2}} \right)^2,
\end{align}
as soon as $K \geq K_0$ and $n \geq n_0$.
We now deal with $T_2$: applying Lemma~\ref{lem:estimates_RW_for_second_moment}\ref{lem:estimates_RW_for_second_moment:it:intermediate s}, we get
\begin{align*} 
	T_2 
	& \leq \sum_{s=1}^{\lceil 3t/4 \rceil - 1}
	\sum_{\ell = i_{s-1}+1}^{i_s}
	\frac{p_\ell^2}{Z_\ell} 
	\cdot \frac{C e^{\theta (s+K)}}{t^3} \frac{(K +1)}{s^{3/2}} \\
	& \leq \frac{C e^{\theta K}}{t^3} (K+1) 
	\left(
	\sum_{\ell = N+1}^{i_1} p_\ell^2
	+
	\sum_{s=2}^{\lceil 3t/4 \rceil - 1}
	\frac{e^{\theta s}}{s^{3/2}}
	\left( \frac{N}{i_{s-1}} \right)^{\gamma(e^{\theta}-1)}
	\sum_{\ell = i_{s-1}+1}^{i_s}
	 p_\ell^2
	 \right),
\end{align*}
using \eqref{eq:bound_Z_n^(N)} when $s \geq 2$ and simply $Z_\ell \geq 1$ in the case $s=1$.
Then, using \eqref{eq:sum_p_i^(N)_square} for both terms
and applying \eqref{eq:estimate_i_k^(N)} to $i_{s-1}$ in the case $s \geq 2$ (bounding the $\kappa_N$ term by a constant independent of $N$), we get
\begin{align} 
	T_2 
	& \leq \frac{C e^{\theta K}}{t^3} (K+1) 
	\left(
	\frac{1}{N}
	+
 	\sum_{s=1}^{\lceil 3t/4 \rceil - 1}
 	\frac{e^{\theta s}}{s^{3/2}}
 	\cdot 
 	\left( e^{-s/\gamma e^{\theta}} \right)^{\gamma(e^{\theta}-1)}
 	\cdot 
 	\frac{e^{-s/\gamma e^{\theta}}}{N}
 	\right) 
 	\leq \frac{C e^{\theta K}}{t^3} \frac{(K+1)}{N}, \label{eq:bound_T_2}
\end{align}
using that $1+\gamma (e^\theta-1-\theta e^\theta) = 0$.
Finally, we deal with $T_3$: applying Lemma~\ref{lem:estimates_RW_for_second_moment}\ref{lem:estimates_RW_for_second_moment:it:large s}, we get
\begin{align*} 
	T_3 
	\leq \sum_{s=\lceil 3t/4 \rceil}^t
	\sum_{\ell = i_{s-1}+1}^{i_s}
	\frac{p_\ell^2}{Z_\ell} 
	\cdot \frac{C e^{\theta (s-x_n)}}{(t-s)^3+1} \frac{(K +1)}{t^{3/2}}
	\leq \frac{C e^{- \theta x_n}}{t^{3/2}} \frac{(K+1)}{N}
	\sum_{s=\lceil 3t/4 \rceil}^t
	\frac{1}{(t-s)^3+1},
\end{align*}
where in the second inequality we proceed as for $T_2$ (note that the sum w.r.t.\@ $\ell$ is identical).
Noting that the sum over $s$ is bounded by a constant and 
$e^{-\theta x_n} \leq C/(\log i_t)^{3/2} \leq C/t^{3/2}$ by \eqref{eq:estimate_i_k^(N)}, it follows that
\begin{align} \label{eq:bound_T_3}
	T_3
	\leq \frac{C}{t^3} \frac{(K+1)}{N}.
\end{align}
Combining \eqref{eq:bound_T_1}, \eqref{eq:bound_T_2} and \eqref{eq:bound_T_3}, we finally get
\begin{align*} \label{eq:final_bound_Q^2}
	\Ec{ Q_n^2 }
	& \leq \frac{Z_n^2}{t^3} \cdot 
	\left( 
	\left( (1+\varepsilon) \sqrt{\frac{2}{\pi}} 
			\frac{(K+ C N^{-c})}{\rho^-} \right)^2
	+  \frac{C(K+1)e^{\theta K}}{N}
	+  \frac{C(K+1)}{N}
	\right).
\end{align*}
This concludes the proof of \eqref{eq:second_moment_Q_n}.

\subsection{Applying the random walk estimates}
\label{subsection:applying_the_RW_estimates}

In this section, we prove Lemma~\ref{lem:estimates_RW_for_second_moment}. For this, we need the following lemma.
\begin{lemma} \label{lem:interactions}
There exist $C,c > 0$ such that, for any integers $t > s \geq 0$, $K\geq 0$, $\ell \geq 1$ and $N \geq K^{2}$, on the event $\{ \forall k \in \intervalleentier{0}{t}, H^\ell_{i_k} - k \leq K \}$, we have
\[
	\sum_{i = i_s+1}^{i_{s+1}} p_i \ind{H^\ell_i \neq H^\ell_{i-1}} 
	\leq C \cdot N^{-1/4} \cdot e^{-c s}.
\]
\end{lemma}
\begin{proof}
We work on the event $\{ \forall k \in \intervalleentier{0}{t}, H^\ell_{i_k} - k \leq K \}$ so we have
		\[
		\sum_{i = i_s+1}^{i_{s+1}} \ind{H^\ell_i \neq H^\ell_{i-1}} 
		\leq H^\ell_{i_{s+1}} 
		\leq s+1 + K.
		\]
Using successively the Cauchy-Schwarz inequality, the last display and \eqref{eq:sum_p_i^(N)_square} we have
\begin{equation*}
	\sum_{i = i_s+1}^{i_{s+1}} p_i \ind{H^\ell_i \neq H^\ell_{i-1}}
	\leq \sqrt{\sum_{i = i_s+1}^{i_{s+1}} \ind{H^\ell_i \neq H^\ell_{i-1}}}\cdot\sqrt{\sum_{i = i_s+1}^{i_{s+1}} p_i^2} \leq \sqrt{s+1 + K} \cdot \frac{C}{\sqrt{i_s\vee N}}.
\end{equation*}
Thanks to \eqref{eq:estimate_i_k^(N)}, which holds for $s\geq 1$, we can write $N\vee i_s\geq C\cdot N\cdot \exp(cs)$, which holds for any $s\geq 0$, for some $c>0$. 
Using the condition that $K\leq N^{\frac12}$ we get
\begin{align*}
\sum_{i = i_s+1}^{i_{s+1}} p_i \ind{H^\ell_i \neq H^\ell_{i-1}} 
\leq C \cdot N^{-\frac14}\cdot \left(\frac{s+1+N^{\frac12}}{N^{\frac12}}\right)^{\frac12}\cdot \exp\left(-\frac{cs}{2}\right)
\leq C\cdot N^{-\frac14} \cdot e^{-cs},
\end{align*}
where, we recall, we allow the values of the constants $C,c>0$ to change along the computation. This finishes the proof of the lemma.
\end{proof}

In the proof of Lemma~\ref{lem:estimates_RW_for_second_moment}, we apply several times the results of Section~\ref{section:RW} to a variety of different random walks. 
All the results of Section~\ref{section:RW} depend on two sequences $(\mathbf r, \mathbf j)$ and in particular the error terms are expressed using the quantities introduced in \eqref{eq:def_delta}. 
In the following lemma, we provide bounds for those error terms that apply uniformly in all the cases that arise in the proof of Lemma~\ref{lem:estimates_RW_for_second_moment}.
\begin{lemma}\label{lem:uniform error bounds}
	There exist $C,c>0$ such that for any integers $t>s\geq 0,\ K\geq 0,\ \ell\geq 1$ being such that $i_s\geq \ell$, for $N\geq \sqrt{K}$, the following inequalities jointly hold for the quantities below defined in \eqref{eq:def_delta} for a family of $(\mathbf{r},\mathbf{j})$ that depends on $s, N,K$, which we describe below 
	\begin{align*}
	\delta_k^{(\mathbf{r},\mathbf{j})}&\leq C\cdot N^{-c}\cdot e^{-c(k+s)}, \quad \text{for all }  1\leq k \leq t-s, \\ \Delta_{t-s}^{(\mathbf{r},\mathbf{j})}&\leq C\cdot N^{-c}\cdot e^{-cs},\\
	\eta_{t-s}^{(\mathbf{r},\mathbf{j})}&\leq C\cdot N^{-c}\cdot e^{-ct}.
	\end{align*}
	The inequalities above hold jointly for $\mathbf{j}=(i_{k+s}-i_s+1)_{k\geq 0}$, which implicitly depends on $N$, and $\mathbf{r}= (p_{i_s-1+i})_{i\geq 2}$ or $ (p^\ell_{i_s-1+i})_{i\geq 2}$ or any realisation of $(\tilde{p}^\ell_{i_s-1+i})_{i\geq 2}$ on the event $\{ \forall k \in \intervalleentier{0}{t}, H^\ell_{i_k} - k \leq K \}$.
\end{lemma}
\begin{proof} Let $s\geq 0$ and $(\mathbf{r},\mathbf{j})$ as in the lemma. For any $1\leq k \leq t-s$, we can write 
\begin{align*}
\delta_k^{(\mathbf{r},\mathbf{j})} 
& = \abs{\Ec{Y_k^{(\mathbf{r},\mathbf{j})}}- 1}= \abs{\sum_{j=j_{k-1}+1}^{j_k} r_j -1} 
\leq  \abs{\sum_{i=i_{k+s-1}+1}^{i_{k+s}} p_i - 1} + \sum_{i=i_{k+s-1}+1}^{i_{k+s}} \abs{r_{i+1-i_s} - p_i}.
\end{align*}
The first term of the last display is bounded above by $(p_{i_{k+s-1}} \1_{k+s-1 \geq 1} + p_{i_{k+s}})$ which is smaller than $C N^{-\alpha} e^{-\alpha(k+s)}$ using \eqref{eq:bound_p_i^(N)}. Then, we consider the different choices of $\mathbf{r}$. 
\begin{itemize}
\item If 
$\mathbf{r}= (p_{i_s-1+i})_{i\geq 2}$, then the second sum is identically equal to $0$.
\item If $\mathbf{r}=(p^\ell_{i_s-1+i})_{i\geq 2}$, then recalling the definition \eqref{eq:definition H^l}, for $i>\ell$ we have
\[\abs{p^\ell_i - p_i} = \abs{\frac{p_i (1-q_i)}{1-p_iq_i}-p_i}= \abs{\frac{p_i q_i (1-p_i)}{1-p_i q_i}}\leq Cp_i^2,\] 
so that using \eqref{eq:sum_p_i^(N)_square} and \eqref{eq:estimate_i_k^(N)} allows us to bound the sum by $C\cdot N^{-1}\cdot  e^{-c(k+s)}$.
\item Last, if $\mathbf{r}$ is any realisation of $(\tilde{p}^\ell_{i_s-1+i})_{i\geq 2}$ on the event $\{ \forall k \in \intervalleentier{0}{t}, H^\ell_{i_k} - k \leq K \}$, recalling that
$\tilde{p}^\ell_i = p_i \ind{H^\ell_i=H^\ell_{i-1}}$ we have 
\[\sum_{i=i_{k+s-1}+1}^{i_{k+s}} \abs{r_{i+1-i_s} - p_i}= \sum_{i=i_{k+s-1}+1}^{i_{k+s}} p_i \ind{H^\ell_i \neq H^\ell_{i-1}} 
\leq C N^{-1/4} e^{-c(k+s)},\]
using Lemma~\ref{lem:interactions}.
\end{itemize}
In the end, by tuning the constants $C,c>0$, we have that in any case, for all $1\leq k \leq t-s$,
\begin{align*}
	\delta_k^{(\mathbf{r},\mathbf{j})}&\leq C\cdot N^{-c} \cdot e^{-cs}.
\end{align*}
From there, it is easy to get that
\begin{align*}
	\Delta_{t-s}^{(\mathbf{r},\mathbf{j})} 
& \leq \sum_{k=1}^{t-s} \delta_k^{(\mathbf{r},\mathbf{j})}\leq  \sum_{k=1}^{t-s} C\cdot N^{-c}\cdot e^{-c(k+s)}
\leq C\cdot N^{-c} \cdot e^{-cs}.
\end{align*}
Then, for any of our choices of $\mathbf{r}$, we have $r_i\leq p_{i_s-1+i}$ for all $i\geq 2$. This allows us to write
\begin{align*}
\eta_{t-s}^{(\mathbf{r},\mathbf{j})}&\leq 2 \left( \sum_{j = i_{\lfloor (t-s)^{1/4} \rfloor+s}+1}^{i_t} p_j^2 + \sum_{j=\lfloor (t-s)^{1/4} \rfloor}^{t-s} \delta_j^{(\mathbf{r},\mathbf{j})} \right)\\
&\leq \frac{C}{i_{\lfloor (t-s)^{1/4} \rfloor+s}} 
+  C N^{-c} e^{-c(\lfloor (t-s)^{1/4} \rfloor+s)}
\leq C\cdot N^{-c}\cdot e^{-ct},
\end{align*}
where we use \eqref{eq:sum_p_i^(N)_square}, \eqref{eq:estimate_i_k^(N)} and our previous estimate on $\delta_j^{(\mathbf{r},\mathbf{j})}$ for $1\leq j \leq t-s$. This finishes the proof of the lemma.
\end{proof}
\begin{proof}[Proof of Lemma~\ref{lem:estimates_RW_for_second_moment}]
We set $E(\ell) \coloneqq \E [ e^{\theta H^\ell_\ell} 
F(H^\ell_1,\dots,H^\ell_n) 
F(\overline{H}_1^\ell, \dots, \overline{H}^\ell_n)]$. 

\textbf{Part~\ref{lem:estimates_RW_for_second_moment:it:large s}}. We consider the case $s \geq 3t/4$. 
Recalling that $\overline{H}_j^\ell = H_j^\ell$ for $j \leq \ell$, we have
\begin{align*}
	E(\ell) 
	& = \Ec{e^{\theta H^\ell_\ell} 
	F(H^\ell_1,\dots,H^\ell_n) \cdot
	\Ppsq{\overline{H}_{i_t}^\ell = t-x_n,\
	\max_{k \in \intervalleentier{s}{t}} \overline{H}_{i_k}^\ell - k \leq -x_n}
	{H^\ell, \overline{H}_{i_s}^\ell}}.
\end{align*}
This last conditional probability is equal to
\begin{align*}
	\Ppsq{\overline{S}_{t-s} = -x_n - \overline{H}_{i_s}^\ell + s,\
	\max_{k \in \intervalleentier{0}{t-s}} \overline{S}_k \leq -x_n-\overline{H}_{i_s}^\ell+s}
	{H^\ell, \overline{H}_{i_s}^\ell},
\end{align*}
where $\overline{S}_k \coloneqq (\sum_{j=1}^k \overline{Y}_{\!\!j}) -k$
with $\overline{Y}_{\!\!j} \coloneqq \sum_{i=i_{j+s-1}+1}^{i_{j+s}} \1_{V_i \leq \tilde{p}^\ell_i}$, recalling that
$\tilde{p}^\ell_i = p_i \ind{H^\ell_i=H^\ell_{i-1}}$ and the $V_i$ are i.i.d.\@ uniformly distributed over $(0,1)$ and independent of $H^\ell$ and $\overline{H}_{i_s}^\ell$.
The distribution of $\overline{S}$ then corresponds to that of $S^{(\mathbf r, \mathbf j)}$, for $\mathbf r=(\tilde{p}^\ell_{i-1+i_s})_{i\geq 2}$ and $\mathbf j=(i_{k+s}-i_s+1)_{k\geq 0}$, in the setting of Section~\ref{section:RW}.
Applying Lemma~\ref{lem:upper_bound_RW_double_barrier} with $K=0$, $a=-x_n - \overline{H}_{i_s}^\ell + s$ and $n=t-s$ to bound this probability, we get that the above display is smaller than 
\begin{align*}
C \left(\Delta_{\lfloor n^{1/4} \rfloor}^{(\mathbf r,\mathbf j)} +1 \right) \frac{(-x_n - \overline{H}_{i_s}^\ell + s+1)}{(t-s)^{3/2}+1}
+ \eta_{t-s}^{(\mathbf r,\mathbf j)}\leq \frac{C (-x_n - \overline{H}_{i_s}^\ell + s+1)}{(t-s)^{3/2}+1},
\end{align*}
where the inequality is due to Lemma~\ref{lem:uniform error bounds}. 
Hence, we have
\begin{align*}
	E(\ell) 
	& \leq \Ec{e^{\theta H^\ell_\ell} 
	F(H^\ell_1,\dots,H^\ell_n)
	\frac{C (-x_n-\overline{H}_{i_s}^\ell+s +1)}{(t-s)^{3/2}+1}} \\
	&\leq \Ec{e^{\theta H^\ell_\ell} 
	F(H^\ell_1,\dots,H^\ell_n) 
	\frac{C (-x_n-H_\ell^\ell+s +1)}{(t-s)^{3/2}+1}},
\end{align*}
where we used that $\overline{H}_{i_s}^\ell \geq H_\ell^\ell$.
Then, setting for brevity $B \coloneqq \{ \max_{k \in \intervalleentier{0}{t/2}} H_{i_k}^\ell - k \leq K,\ \max_{k \in \intervalleentier{t/2}{s-1}} H_{i_k}^\ell - k \leq -x_n \}$, we have
\begin{align*}
	E(\ell) 
	\leq \Ec{ e^{\theta H^\ell_\ell} \1_B \cdot
	\frac{C (-x_n-H_\ell^\ell+s+1)}{(t-s)^{3/2}+1} 
	\cdot
	\Ppsq{
	\max_{k \in \intervalleentier{s}{t}} H_{i_k}^\ell - k \leq -x_n = H_{i_t}^\ell - t}
	{H^\ell_{i_s}}
	}.
\end{align*}
We apply Lemma~\ref{lem:upper_bound_RW_double_barrier} again to bound the conditional probability appearing in the last display.
In that case the considered random walk is 
$\widetilde{S}_k \coloneqq (\sum_{j=1}^k \widetilde{Y}_j) -k$
with $\widetilde{Y}_j \coloneqq \sum_{i=i_{j+s-1}+1}^{i_{j+s}} \1_{U_i \leq p^\ell_i}$, recalling that
$p^\ell_i = p_i (1-q_i)/(1-p_iq_i)$. 
The distribution of $\widetilde{S}$ then corresponds to that of $S^{(\mathbf r, \mathbf j)}$, for $\mathbf r=(p^\ell_{i-1+i_s})_{i\geq 2}$ and $\mathbf j=(i_{k+s}-i_s+1)_{k\geq 0}$, and the conditional probability above can be written as
\begin{align*}
\Ppsq{\widetilde{S}_{t-s} = -x_n - H_{i_s}^\ell + s,\
	\max_{k \in \intervalleentier{0}{t-s}} \widetilde{S}_k \leq -x_n-H_{i_s}^\ell+s}
{H_{i_s}^\ell}.
\end{align*}
Applying again Lemma~\ref{lem:upper_bound_RW_double_barrier} and Lemma~\ref{lem:uniform error bounds} as before, it follows that
\begin{align*}
	E(\ell)
	& \leq \Ec{e^{\theta H^\ell_\ell} 
	\1_B \cdot 
	\frac{C(-x_n-H_\ell^\ell+s +1)}{(t-s)^{3/2}+1} \cdot 
	\frac{(-x_n-H_{i_s}^\ell+s +1)}{(t-s)^{3/2}+1} } \\
	& \leq \Ec{ e^{\theta H^\ell_{i_{s-1}}} 
	\1_B \cdot
	\frac{(-x_n-H_{i_{s-1}}^\ell+s +1)^2}{(t-s)^3+1} },
\end{align*}
using that $H_{i_s}^\ell \geq H_\ell^\ell \geq H_{i_{s-1}}^\ell$ and that
$\E \bigl[ e^{\theta (H^\ell_\ell-H^\ell_{i_{s-1}})} \bigr] \leq C$.
Finally, we apply Lemma~\ref{lem:exponential_RW_double_barrier} to the random walk $S_k \coloneqq (\sum_{j=1}^k Y_j) -k$
with $Y_j \coloneqq \sum_{i=i_{j-1}+1}^{i_j} \1_{U_i \leq p_i}$, and we get
\begin{align*}
	E(\ell)
	& \leq \frac{C}{(t-s)^3+1}
	e^{\theta (s-1-x_n)} \frac{(K +1)}{(s-1)^{3/2}}.
\end{align*}
This concludes the proof of Part~\ref{lem:estimates_RW_for_second_moment:it:large s}.

\textbf{Part~\ref{lem:estimates_RW_for_second_moment:it:intermediate s}}.
First note that $E(\ell)$ is smaller than
\begin{align*}
	\E \Bigl[ e^{\theta H^\ell_\ell} 
	F(H^\ell_1,\dots,H^\ell_n) 
	\cdot
	\ind{\overline{H}_{i_t}^\ell = t-x_n,\ 
	\max_{k \in \intervalleentier{s}{(s+t)/2}} \overline{H}_{i_k}^\ell - k \leq K,\
	\max_{k \in \intervalleentier{(s+t)/2}{t}} \overline{H}_{i_k}^\ell - k \leq -x_n}
	\Bigr],
\end{align*}
where we replaced the barrier at $-x_n$ by a barrier at $K$ at some points.
We integrate w.r.t.\@ the random walk $\overline{S}$ as before, it follows from Lemma~\ref{lem:uniform error bounds} and  Lemma~\ref{lem:upper_bound_RW_double_barrier} that
\begin{align*}
	E(\ell) 
	& \leq \Ec{e^{\theta H^\ell_\ell} 
	F(H^\ell_1,\dots,H^\ell_n) \cdot
	\frac{C (K-H_\ell^\ell+s +1)}{(t-s)^{3/2}+1}}.
\end{align*}
Then, we integrate w.r.t.\@ the random walk $S$ similarly and get, by  Lemma~\ref{lem:uniform error bounds} and  Lemma~\ref{lem:upper_bound_RW_double_barrier},
\begin{align*}
	E(\ell) 
	& \leq \Ec{ e^{\theta H^\ell_{i_{s-1}}} 
		\ind{\max_{k \in \intervalleentier{0}{s}} H_{i_k}^\ell - k \leq K}
		\frac{C(K-H_{i_{s-1}}^\ell+s +1)^2}{(t-s)^3+1} }.
\end{align*}
Finally, applying Lemma~\ref{lem:exponential_RW_double_barrier}, we get the announced result.

\textbf{Part~\ref{lem:estimates_RW_for_second_moment:it:first term}}. We are in the case $\ell = 1$, so $H^\ell_\ell = 0$ and $E(\ell)$ equals
\begin{align*}
	\Ec{F(H^\ell_1,\dots,H^\ell_n)
	\Ppsq{\overline{H}_{i_t}^\ell = t-x_n,
	\max_{k \in \intervalleentier{0}{t/2}} \overline{H}_{i_k}^\ell - k \leq K,
	\max_{k \in \intervalleentier{t/2}{t}} \overline{H}_{i_k}^\ell - k \leq -x_n}
	{H^\ell}}.
\end{align*}
Then, we apply Lemma~\ref{lem:estimate_RW_double_barrier} to the random walk $\overline{S}$, noting that $R^-(0) = 1$, to get
\begin{align*}
	E(\ell)
	\leq (1+\varepsilon) \sqrt{\frac{2}{\pi}} 
		\frac{(K+ C N^{-c})}{\rho^- t^{3/2}}
	\Ec{F(H^\ell_1,\dots,H^\ell_n)}.
\end{align*}
Applying Lemma~\ref{lem:estimate_RW_double_barrier} to the random walk $S$, the result follows.
\end{proof}

\section{Diameter of the tree}
\label{section:diameter}

\begin{proof}[Proof of Theorem~\ref{thm:diameter_WRT}]
First note that we have $\diam(\ttT_n) \leq 2 \haut(\ttT_n)$ so the upper bound follows directly from Theorem~\ref{thm:height_WRT}.
Now we fix some $\varepsilon >0$ and we want to prove that there exists $b \in \R$ such that
\begin{align*}
	\limsup_{n\to\infty}
	\Pp{\diam(\ttT_n) \geq 2 \gamma e^{\theta} \log n - \frac{3}{\theta} \log \log n - b} 
	\geq 1-\varepsilon.
\end{align*}
For this, we use notation and results from Section~\ref{section:lower_bound}.
By an argument similar to the proof of the lower bound in Theorem~\ref{thm:height_WRT}, it is enough to prove that for $N$ large enough
\begin{align*}
	\limsup_{t\to\infty} \Pp{\diam(\ttT_n^{(N)}) \geq 2t - 2x_n}
	\geq 1-\varepsilon,
\end{align*}
where $n \coloneqq i_t^{(N)}$ and $\ttT_n^{(N)}$ has distribution $\wrt(\boldsymbol{w}^{(N)})$, where the sequence of weights $\boldsymbol w^{(N)}$ is defined in \eqref{eq:def_w^(N)}.
In the rest of this proof, we work only with the tree $\ttT_n^{(N)}$ for some fixed $N$ that is chosen large enough depending on $\varepsilon$ afterwards.
Therefore, from now on, we omit the dependence in $N$ in the notation of the various quantities we are considering (including $w_m^{(N)}$).
Recall that, for some $K \geq 0$, we consider
\begin{align*}
	Q_n
	& \coloneqq \sum_{m=1}^{n} \frac{w_m}{W_n} 
	e^{\theta \haut(\ttu_m)} 
	\1_{B_m}, \\
 	B_m 
 	& \coloneqq
	\left\{ \haut(\ttu_m) = t - x_n,\ 
	\max_{k \in \intervalleentier{0}{t/2}} \haut(\ttu_m(i_k)) - k \leq K,\ 
	\max_{k \in \intervalleentier{t/2}{t}} \haut(\ttu_m(i_k)) - k \leq -x_n \right\}.
\end{align*}
Observe that if there are two vertices in $\ttT_n^{(N)}$ at height $t - x_n$ whose most recent common ancestor is the root, then the diameter of $\ttT_n^{(N)}$ is at least $2t - 2x_n$.
Hence, recalling that $\ttu \wedge \ttv$ denotes the most recent common ancestor of vertices $\ttu$ and $\ttv$, we have
\begin{align*}
	\Pp{\diam(\ttT_n^{(N)}) \geq 2t - 2x_n}
	& \geq \Pp{ \exists \ttu,\ttv \in \ttT_n^{(N)}: 
	\haut(\ttu) = \haut(\ttv) = t - x_n,\
	\ttu \wedge \ttv = \ttu_1} \\
	&\geq \Pp{\cT_1 > 0},
\end{align*}
where we set
\begin{align*}
	\cT_1
	& \coloneqq \sum_{\ell,m=1}^{n} 
	\frac{w_\ell w_m}{(W_n)^2} 
	e^{\theta \haut(\ttu_\ell)} e^{\theta \haut(\ttu_m)} 
	\1_{B_\ell} \1_{B_m} 
	\ind{\mathrm{lab}(\ttu_\ell \wedge \ttu_m) = 1}.
\end{align*}
Note that $\cT_1$ is a part of the sum obtained when developing $Q_n^2$ and the remaining part satisfies
\begin{align*}
	\Ec{Q_n^2- \cT_1}
	& = \Ec{\sum_{\ell,m=1}^{n} 
		\frac{w_\ell w_m}{(W_n)^2} 
		e^{\theta \haut(\ttu_\ell)} e^{\theta \haut(\ttu_m)} 
		\1_{B_\ell} \1_{B_m} 
		\ind{2 \leq \mathrm{lab}(\ttu_\ell \wedge \ttu_m) \leq n} }
	\leq T_2 + T_3,
\end{align*}
by the many-to-two lemma (Lemma~\ref{lem:many-to-two}),
where $T_2$ and $T_3$ were defined in Section~\ref{subsection:second_moment_Q} as parts of the sum on the right-hand of \eqref{eq:first_bound_second_moment_Q_n} corresponding to $2 \leq \ell \leq i_{\lceil 3t/4 \rceil - 1}$ and $i_{\lceil 3t/4 \rceil - 1} < \ell \leq n$ respectively.
Then, we proved in \eqref{eq:bound_T_2} and \eqref{eq:bound_T_3} that
\begin{align*}
	T_2 + T_3 
	\leq \frac{Z_n^2}{t^3} \cdot \frac{C(K+1)e^{\theta K}}{N}
	\leq \varepsilon \Ec{Q_n}^2,
\end{align*}
where the second inequality follows from \eqref{eq:first_moment_Q_n} for $K,N,n$ large enough depending on $\varepsilon$ only.
Therefore, we get
\begin{align*}
	\Pp{\cT_1 > 0}
	& \geq \Pp{Q_n > \frac{1}{2} \Ec{Q_n},\ 
	Q_n^2- \cT_1 < \frac{1}{4} \Ec{Q_n}^2} \\
	& \geq 1 
	- \Pp{Q_n \leq \frac{1}{2} \Ec{Q_n} }
	- \Pp{Q_n^2 - \cT_1 \geq \frac{1}{4} \Ec{Q_n}^2} \\
	& \geq 1
	- \frac{4 \Var \left(Q_n\right)}{\Ec{Q_n}^2}
	- 4 \varepsilon,
\end{align*}
applying Chebyshev and Markov inequalities.
By Lemma~\ref{lem:first_and_second_moment}, we have $\Var(Q_n) \leq \varepsilon \E[Q_n]^2$ for $K,N,n$ large enough depending on $\varepsilon$ only.
Hence, we proved that for $K,N,n$ large enough,
\begin{align*}
	\Pp{\diam(\ttT_n^{(N)}) \geq 2t - 2x_n}
	\geq \Pp{\cT_1 > 0}
	\geq 1 - 6 \varepsilon,
\end{align*}
which concludes the proof.
\end{proof}

\begin{appendix}
\section{Random walk estimates}
\label{section:RW}

The goal of this section is to prove estimates for the probability of events involving a certain inhomogeneous random walk $(S_k)$.
We work in the following framework: let $\mathbf{r}=(r_i)_{i\geq 2}$ be a sequence of real numbers in the interval $\intervalleff{0}{1}$.
Then, let $\mathbf{j}=(j_k)_{k \geq 0}$ be an increasing sequence of integers with $j_0 = 1$. 
We introduce the following processes that depend on $\mathbf{r}$ and~$\mathbf{j}$
\begin{align*}
	Y_k^{(\mathbf{r},\mathbf{j})} & \coloneqq \sum_{j=j_{k-1}+1}^{j_k} \ind{U_j \leq r_j}, 
	\qquad \text{for } k \geq 1, \\
	S_k^{(\mathbf{r},\mathbf{j})} & \coloneqq \left( \sum_{\ell=1}^k Y_\ell^{(\mathbf{r},\mathbf{j})} \right) - k,
	\qquad \text{for } k \geq 0,
\end{align*}
where $(U_j)_{j\geq2}$ is a sequence of i.i.d.\@ uniform random variable over $(0,1)$.
Finally, we define
\begin{align}
\begin{split} \label{eq:def_delta}
	& \delta_k^{(\mathbf{r},\mathbf{j})} \coloneqq \abs{\Ec{Y_k^{(\mathbf{r},\mathbf{j})}}-1},
	\qquad 
	\Delta_k^{(\mathbf{r},\mathbf{j})}
	\coloneqq \max_{0 \leq \ell \leq k} \abs{\Ec{S_\ell^{(\mathbf{r},\mathbf{j})}}}, \\
	& \eta_k^{(\mathbf{r},\mathbf{j})} \coloneqq 2 \left( \sum_{j = j_{\lfloor k^{1/4} \rfloor}+1}^{j_k} r_j^2 \right) + 2 \left( \sum_{\ell=\lfloor k^{1/4} \rfloor}^k \delta_\ell^{(\mathbf{r},\mathbf{j})} \right).
\end{split}
\end{align}
which are non-negative numbers appearing in error terms.
Throughout the paper, we make use of the estimates proved in this section for several choices of $(\mathbf{r},\mathbf{j})$. 
In particular, for a fixed $(\mathbf{r},\mathbf{j})$, it is useful to apply the results for the walk $(S_{k+s}^{(\mathbf{r},\mathbf{j})}-S_{s}^{(\mathbf{r},\mathbf{j})})_{k\geq 0}$ which has the same distribution as $(S_k^{(\mathbf{r}',\mathbf{j}')})_{k\geq 0}$ where $\mathbf{r}'=(r_{j_s+i-1})_{i\geq 2}$ and $\mathbf{j}'=(j_{s+k}-j_s+1)_{k\geq 0}$. 
In this section, we are going to make the dependency in  $(\mathbf{r},\mathbf{j})$ implicit because those sequences are chosen in different ways throughout the paper.

\subsection{A coupling with an homogeneous random walk}

The goal of this section is to prove the following lemma, which allows us to apply known results on homogeneous random walks.
\begin{lemma} \label{lem:comparison_S_Stilde}
For any $m \geq 0$, there exists a random walk $\widehat{S}$ with jump distribution $\mathrm{Poisson}(1) - 1$ such that
\[
	\Pp{ \exists k \in \llbracket 0,n-m \rrbracket : S_{k+m}-S_m \neq \widehat{S}_k }
	\leq 2 \left( \sum_{j = j_m+1}^{j_n} r_j^2 \right)
	+ 2 \left( \sum_{k=m}^n \delta_k \right).
\]
\end{lemma}
It is proved easily by replacing each $Y_\ell$ by a $\mathrm{Poisson}(1)$ r.v.\@ using the following lemma.
\begin{lemma} \label{lem:comparison_Y_Poisson}
Let $q_1,\dots,q_n$ be non-negative real number, $V_1,\dots,V_n$ be independent r.v.\@ uniformly distributed over $(0,1)$ and $Y \coloneqq \sum_{i=1}^n \ind{V_i \leq q_i}$.
There exists a r.v.\@ $Z$ with distribution $\mathrm{Poisson}(1)$ such that 
\[
	\Pp{Y \neq Z}
	\leq 2 \left( \sum_{i=1}^n q_j^2 \right)
	+ 2 \abs{\E[Y] - 1}.
\]
\end{lemma}
\begin{proof} On the one hand, it follows from \cite[Proposition~1]{lecam1960} that the total variation distance between the distribution of $Y$ and the distribution $\mathrm{Poisson}(\E[Y])$ is at most $\sum_{i=1}^n q_i^2$.
On the other hand, by \cite[Equation (2.2)]{adelljodra2006}, the total variation distance between $\mathrm{Poisson}(\E[Y])$ and  $\mathrm{Poisson}(1)$ is at most $\abs{\E[Y] - 1}$.
The result follows.
\end{proof}

\subsection{Known results on the homogeneous random walk}
\label{subsection:known_results_RW}

In this section, we state some known results concerning homogeneous random walks. We work in the particular case of the walk $\widehat{S}$, which jumps with distribution $\mathrm{Poisson}(1) - 1$. Hence we are in the so-called lattice case, because the walk $\widehat{S}$ can take only integer values.

We first introduce $\mathrm R$ the renewal function of the first strict ascending ladder height process of the random walk $\widehat{S}$.
For $x \geq 0$,
\begin{align*}
\mathrm R(x) 
\coloneqq
\sum_{k = 0}^\infty
\Pp{H_k \leq x},
\end{align*}
where $(H_k)_{k\in\N}$ is the first strict ascending ladder height process: we set $\tau_0 \coloneqq 0$, $H_0 \coloneqq 0$ and, for $k \geq 1$, $\tau_k \coloneqq \inf \{ n > \tau_{k-1} : \widehat{S}_n > \widehat{S}_{\tau_{k-1}} \}$ and $H_k \coloneqq \widehat{S}_{\tau_k}$.

Since $\E[\widehat{S}_1] = 0$ and $\E[(\widehat{S}_1)^2] < \infty$, by Feller \cite[Theorem~XVIII.5.1 (5.2)]{feller1971}, we have $\E[H_1] < \infty$.
Thus, it follows from Feller's \cite[p.\@ 360]{feller1971} renewal theorem that there exists a constant $\rho > 0$ such that
\begin{equation} \label{eq:equivalent_function_R}
	\frac{\mathrm R(x)}{x} 
	\xrightarrow[x \to \infty]{} 
	\rho.
\end{equation}
Moreover, we denote by $\mathrm R^-$ the renewal function of the first strict ascending ladder height process for the random walk with jump $1-\mathrm{Poisson}(1)$ and by $\rho^-$ the constant such that $\mathrm R^-(x) / x \to \rho^-$ as $x \to \infty$.

We now state a result which is a direct corollary of \cite[Proposition~2.8]{pain2018}.
Let $(\gamma_n)_{n\in \N}$ be a sequence of positive numbers such that $\gamma_n = o(\sqrt{n})$ as $n\to \infty$.
Then, for all $\lambda \in (0,1)$,
\begin{align} \label{eq:estimate_RW_double_barrier}
	\Pp{\max_{k \leq \lfloor \lambda n \rfloor} \widehat{S}_k \leq K,\ 
		\max_{\lfloor\lambda n\rfloor \leq i \leq n} \widehat{S}_i \leq L,\
		\widehat{S}_n = L - a} 
	= \sqrt{\frac{2}{\pi}} 
	\frac{\mathrm R(K) \mathrm R^-(a)}{n^{3/2} \rho \rho^-}
	(1 + \petito{1}),
\end{align}
as $n \to \infty$, uniformly in $K \in [0,\gamma_n]$, $L \in [-\gamma_n, \gamma_n]$ and $a \in [0, \gamma_n] \cap (L + \Z)$.

Moreover, Lemma~2.4 of A\"idékon and Shi \cite{aidekonshi2014} shows the following upper bound: for $\lambda \in (0,1)$,
there exists $C >0$ depending on $\lambda$ such that for all $a \geq 0$, $K \geq 0$, $L \in\R$ and $n \geq 0$, we have
\begin{align} \label{eq:upper_bound_RW_double_barrier}
	\Pp{\max_{k \leq \lfloor \lambda n \rfloor} \widehat{S}_k \leq K,\ 
		\max_{\lfloor\lambda n\rfloor \leq i \leq n} \widehat{S}_i \leq L,\
		\widehat{S}_n = L - a} 
	\leq
	\frac{C(K+1)(a+1)}{n^{3/2}+1}.
\end{align}

\subsection{Estimates on random walk \texorpdfstring{$S$}{S}}

\begin{lemma} \label{lem:Bernoulli_rv}
Let $Y$ be a sum of independent Bernoulli random variables.
Then for any integer $b \geq 0$, 
\[
	\Ec{ \ind{Y-b \geq 1} (Y-b)}
	\leq \Pp{Y-b \geq 1} \Ec{Y+1}
\]
\end{lemma}
\begin{proof} By assumption, $Y$ is of the form $\sum_{i=1}^n B_i$, where the $B_i$'s are independent Bernoulli r.v. Then let $T \coloneqq \inf \{ k \geq 0 : \sum_{i=1}^k B_i = b+1 \}$, where $\inf \emptyset = \infty$.
\begin{align*}
	\Ec{ \ind{Y-b \geq 1} (Y-b)}
	& = \Ec{ \ind{T \leq n} (Y-b)}
	= \sum_{k=1}^n \Ec{ \ind{T = k} \left( 1 + \sum_{i=k+1}^n B_i \right)} \\
	& = \sum_{k=1}^n \Pp{T = k} \Ec{1 + \sum_{i=k+1}^n B_i}
	\leq \Pp{T \leq n} \Ec{1 + Y},
\end{align*}
and it proves the result.
\end{proof}
\begin{lemma} \label{lem:expectation_beginning_RW} 
For any $\varepsilon > 0$, there exists $K_0 > 0$ that does not depend on $(\mathbf r,\mathbf j)$ such that, for $K \geq K_0$, for any $m \geq 0$, we have
\[
	(1-\varepsilon) \rho (K-\Delta_m)
	\leq \Ec{ \ind{\max_{j \leq m} S_j \leq K} \mathrm R(K-S_m)}
	\leq (1+\varepsilon) \rho \left( K+2\Delta_m\right).
\]
\end{lemma}
\begin{proof}
Let $\varepsilon > 0$ be fixed.
For $K$ large enough, by \eqref{eq:equivalent_function_R}, we have for any $x \geq \sqrt{K}$,
\[
	(1-\varepsilon) \rho x 
	\leq \mathrm R(x)
	\leq (1+\varepsilon) \rho x.
\]
Then, distinguishing between the case $S_m > K-\sqrt{K}$ and $S_m \leq K-\sqrt{K}$, we get
\begin{align*}
	\Ec{ \ind{\max_{j \leq m} S_j \leq K} \mathrm R(K-S_m)}
	& \leq \mathrm R(\sqrt{K})
	+ (1+\varepsilon) \rho \Ec{ \ind{\max_{j \leq m} S_j \leq K} (K-S_m) \ind{S_m \leq K-\sqrt{K}}} \\
	& \leq \varepsilon K 
	+ (1+\varepsilon) \rho \Ec{ \ind{\max_{j \leq m} S_j \leq K}
	 (K-S_m)},
\end{align*}
for $K$ large enough using \eqref{eq:equivalent_function_R} again.
Proceeding similarly, we have
\begin{align*}
	\Ec{ \ind{\max_{j \leq m} S_j \leq K} \mathrm R(K-S_m)}
	&\geq \Ec{ \ind{\max_{j \leq m} S_j \leq K} \mathrm R(K-S_m) \ind{S_m\leq K-\sqrt{K}}}\\
	&\geq (1-\epsilon) \rho \Ec{\ind{\max_{j \leq m} S_j \leq K} (K-S_m)\ind{S_m\leq K-\sqrt{K}}}\\
	& \geq -\varepsilon K 
	+ (1-\varepsilon) \rho \Ec{ \ind{\max_{j \leq m} S_j \leq K} (K-S_m)}.
\end{align*}
Hence, it is now sufficient to prove the following bounds
\begin{align} \label{eq:encadrement_but}
	K-\Delta_m \leq \Ec{ \ind{\max_{j \leq m} S_j \leq K} (K-S_m)}
	\leq K + 2 + 2\Delta_m.
\end{align}
For this, we write
\begin{align*}
	\Ec{ \ind{\max_{j \leq m} S_j \leq K} (K-S_m)} 
	&= - \E_{-K} \left[ \ind{\max_{j \leq m} S_j \leq 0} S_m \right]
	= - \E_{-K} \left[ S_{\tau \wedge m} \right] 
	+ \E_{-K} \left[ \ind{\tau \leq m} S_{\tau} \right],
\end{align*}
where $\tau \coloneqq \inf \{ k \geq 0 : S_k > 0 \}$.
Recall that, for any $k \leq m$, $\abs{\E[S_k]} \leq \Delta_m$.
Hence, applying the optimal stopping theorem to the martingale $(S_k - \E[S_k])$ under $\P_{-K}$, we get that $-K-\Delta_m \leq \E_{-K} [S_{\tau \wedge m}] \leq -K+\Delta_m$.
Thus, \eqref{eq:encadrement_but} follows from the bounds
\begin{align} \label{eq:encadrement_but_2}
	0 \leq \E_{-K} \left[ \ind{\tau \leq m} S_\tau \right] \leq 2 + \Delta_m.
\end{align}
The lower bound in \eqref{eq:encadrement_but_2} holds because $S_\tau \geq 0$.
For the upper bound, we distinguish according to the value of $\tau$:
\begin{align*}
	\E_{-K} \left[ \ind{\tau \leq m} S_\tau \right]
	& = \sum_{k=1}^m \E_{-K} \left[ \ind{S_1,\dots,S_{k-1} \leq 0} \ind{S_k \geq 1} S_k \right] \\
	& = \sum_{k=1}^m \E_{-K} \left[ \ind{S_1,\dots,S_{k-1} \leq 0} 
	\Ecsq{ \ind{Y_k-1+S_{k-1} \geq 1} (Y_k-1+S_{k-1}) }{S_{k-1}}
	\right] \\
	& \leq \sum_{k=1}^m \E_{-K} \left[ \ind{S_1,\dots,S_{k-1} \leq 0} 
		\Ppsq{Y_k-1+S_{k-1} \geq 1}{S_{k-1}} \right]
		(\Ec{Y_k}+1), 
\end{align*}
applying Lemma~\ref{lem:Bernoulli_rv}. Writing $\Ec{Y_k} + 1 = 2 + \Ec{Y_k-1}$, we finally get
\begin{align*}
	\E_{-K} \left[ \ind{\tau \leq m} S_\tau \right]
	& \leq 2 \sum_{k=1}^m \P_{-K} \left( \tau=k \right)
	+ \sum_{k=1}^m \Ec{Y_k-1}
	\leq 2 + \Delta_m.
\end{align*}
This proves the upper bound in \eqref{eq:encadrement_but_2} and hence conclude the proof of the lemma.
\end{proof}
\begin{lemma} \label{lem:estimate_RW_double_barrier}
For any $\varepsilon > 0$ and $\lambda \in (0,1)$, there exist $K_0 > 0$ and $n_0 \geq 1$ that do not depend on $(\mathbf r,\mathbf j)$ such that, for any $n \geq n_0$, any $K \in \llbracket K_0, n^{1/4} \rrbracket$, any $L \in \llbracket -n^{1/4}, n^{1/4} \rrbracket$, and any $a \in \llbracket 0, n^{1/4} \rrbracket$, we have
\begin{align*}
	(1-\varepsilon) \sqrt{\frac{2}{\pi}} 
	\left( K-\Delta_{\lfloor n^{1/4} \rfloor} \right)
	\frac{\mathrm R^-(a)}{\rho^- n^{3/2}}
	- \eta_n 
	& \leq \Pp{S_n = L-a,
	\max_{k < \lambda n} S_k \leq K,
	\max_{\lambda n \leq k \leq n} S_k \leq L} \\
	& \leq (1+\varepsilon) \sqrt{\frac{2}{\pi}} 
	\left( K+2\Delta_{\lfloor n^{1/4} \rfloor} \right)
	\frac{\mathrm R^-(a)}{\rho^- n^{3/2}}
	+ \eta_n.
\end{align*}
\end{lemma}
\begin{proof}
We set $m \coloneqq \lfloor n^{1/4} \rfloor$.
We apply Markov's property at time $m$ and get
\begin{align} \label{eq:markov_at_m}
	\Pp{S_n = L-a,\
	\max_{k < \lambda n} S_k \leq K,\
	\max_{\lambda n \leq k \leq n} S_k \leq L}
	& = \Ec{\ind{\max_{j \leq m} S_j \leq K} \psi(S_m) },
\end{align}
where we set
\[
	\psi(x) \coloneqq \P \Big( S_n - S_m = L-a-x,\
		 \max_{k\in < \lambda n-m} S_{m+k} - S_m \leq K-x,
		 \max_{\lambda n-m \leq k \leq n-m} S_{m+k} - S_m \leq L-x\Big).
\]
Applying Lemma~\ref{lem:comparison_S_Stilde}, we have
\[
	\Pp{\exists k \in \llbracket 0,n-m \rrbracket : \widehat{S}_k \neq S_{m+k} - S_m}
	\leq 2 \left( \sum_{j = j_m+1}^{j_n} r_j^2 \right) + 2 \left( \sum_{k=m}^n \delta_k \right)
	= \eta_n.
\]
Hence, for any $x \geq 0$, we have $\widehat{\psi}(x) - \eta_n \leq \psi(x) \leq \widehat{\psi}(x) + \eta_n$, where we set
\begin{align*}
	\widehat{\psi}(x) 
	& \coloneqq \Pp{ \widehat{S}_{n-m} = L-a-x,\
		\max_{k < \lambda n-m} \widehat{S}_k \leq K-x,\
		\max_{\lambda n-m \leq k \leq n-m]} \widehat{S}_k \leq L-x}.
\end{align*}
Applying \eqref{eq:estimate_RW_double_barrier}, there exists $n_0 \geq 1$, such that for any $n \geq n_0$, any $a,K \in \llbracket 0, n^{1/4} \rrbracket$, any $L \in \llbracket -n^{1/4}, n^{1/4} \rrbracket$ and any $x \in \llbracket -n^{1/4}, K \rrbracket$, 
\begin{align*}
	(1-\varepsilon) \sqrt{\frac{2}{\pi}} \frac{1}{\rho \rho^-}
		\frac{\mathrm R(K-x) \mathrm R^-(a)}{n^{3/2}}
	\leq \widehat{\psi}(x) 
	\leq (1+\varepsilon) \sqrt{\frac{2}{\pi}} \frac{1}{\rho \rho^-}
		\frac{\mathrm R(K-x) \mathrm R^-(a)}{n^{3/2}}.
\end{align*}
Coming back to \eqref{eq:markov_at_m}, we can apply the above with $x = S_m$, because we are on the event $\{ S_m \leq K \}$ and the inequality $S_m \geq -m$ always holds by definition.
Hence, we get the upper bound
\begin{align*}
	\Ec{\ind{\max_{j \leq m} S_j \leq K} \psi(S_m) }
	\leq (1+\varepsilon) \sqrt{\frac{2}{\pi}} \frac{1}{\rho \rho^-}
			\frac{\mathrm R^-(a)}{n^{3/2}}
		\Ec{\ind{\max_{j \leq m} S_j \leq K} \mathrm R(K-S_m) }
		+ \eta_n
\end{align*}
and a similar lower bound holds with $-\varepsilon$ and $-\eta_n$ instead of $\varepsilon$ and $\eta_n$. Applying Lemma~\ref{lem:expectation_beginning_RW} (which determines the choice of $K_0$), it concludes the proof.
\end{proof}
\begin{lemma} \label{lem:upper_bound_RW_double_barrier}
For any $\lambda \in (0,1)$, there exists $C > 0$ that does not depend on $(\mathbf r,\mathbf j)$  such that, for any $n \geq 0$, any $K,a \geq 0$ and any $L \in \R$, we have
\[
	\Pp{S_n = L-a,
	\max_{k < \lambda n} S_k \leq K,
	\max_{\lambda n \leq k \leq n} S_k \leq L}
	\leq C \left( K+\Delta_{\lfloor n^{1/4} \rfloor} +1 \right) \frac{a+1}{n^{3/2}+1}
	+ \eta_n.
\]
The constant $C$ can be chosen uniformly for $\lambda$ in a compact subset of $(0,1)$.
\end{lemma}
\begin{proof} This lemma is proved similarly as Lemma~\ref{lem:estimate_RW_double_barrier}, using \eqref{eq:upper_bound_RW_double_barrier} instead of \eqref{eq:estimate_RW_double_barrier} and the upper bound in \eqref{eq:encadrement_but} instead of Lemma~\ref{lem:expectation_beginning_RW}.
The fact that the constant $C$ can be chosen uniformly for $\lambda$ in a compact subset of $(0,1)$ follows from the observation that the considered probability is nondecreasing in $\lambda$ if $L \leq K$, and nonincreasing in $\lambda$ otherwise.
\end{proof}
\begin{lemma} \label{lem:exponential_RW_double_barrier}
For any $\lambda \in (0,1)$ and $z > 0$, there exists $C > 0$ that does not depend on $(\mathbf r,\mathbf j)$ such that, for any integers $n,K \geq 0$ and $L \in \Z$, we have
\begin{align*}
	\Ec{ \ind{\max_{k < \lambda n} S_k \leq K,\
	\max_{\lambda n \leq k \leq n} S_k \leq L}
	e^{z S_n} (L-S_n+1)^2 }
	& \leq C e^{z L} \left( 
	\frac{K + \Delta_{\lfloor n^{1/4} \rfloor} +1}{n^{3/2}+1}
	+ \eta_n \right).
\end{align*}
The constant $C$ can be chosen uniformly for $\lambda$ in a compact subset of $(0,1)$.
\end{lemma}
\begin{proof} 
We distinguish according to the value of $S_n$:
\begin{align*}
	& \Ec{ \ind{\max_{k\in[0,\lambda n)} S_k \leq K,
	\max_{k\in[\lambda n,n]} S_k \leq L}
	e^{z S_n} (L-S_n+1)^2} \\
	& = \sum_{a = 0}^\infty (a+1)^2 e^{z (L-a)} 
	\Pp{\max_{k < \lambda n} S_k \leq K,
	\max_{\lambda n \leq k \leq n} S_k \leq L,
	S_n = L-a} \\
	& \leq \sum_{a = 0}^\infty (a+1)^2 e^{z (L-a)} 
	\left( C \frac{(K+ \Delta_{\lfloor n^{1/4} \rfloor}+1)(a+1)}{n^{3/2}+1} + \eta_n \right),
\end{align*}
applying Lemma~\ref{lem:upper_bound_RW_double_barrier}.
The result follows.
\end{proof}

\section{Concerning assumptions for preferential attachment trees}
\label{section:PAT}

%

\begin{proof}[Proof of Lemma~\ref{lem:remainder sum p_i square for PAT}]
	Recall the formulas for the $q$-th moment of a Beta distribution: if $\beta\sim \mathrm{Beta}(a,b)$ then
	\begin{align*}
	\Ec{\beta^q}=\prod_{k=0}^{q-1}\frac{a+k}{a+b+k}.
	\end{align*}
	Recall also that if $\beta\sim\mathrm{Beta}(a,b)$, then $(1-\beta)\sim \mathrm{Beta}(b,a)$. 
	
	For this proof, we write $Z_i \coloneqq (\frac{\mathsf{w}^\mathbf{a}_i}{\mathsf{W}^\mathbf{a}_i})^2$ for every $i\geq 2$.
	Using the definition of the sequence $(\mathsf{w}^\mathbf{a}_n)_{n\geq 1}$, we can write for any $i\geq 1$,
	\begin{align*}
	Z_i=\left(\frac{\mathsf{w}^\mathbf{a}_{i}}{\mathsf{W}^\mathbf{a}_{i}}\right)^2= \left(\frac{\mathsf{W}^\mathbf{a}_{i}-\mathsf{W}^\mathbf{a}_{i-1}}{\mathsf{W}^\mathbf{a}_{i}}\right)^2=(1-\beta_{i-1})^2, 
	\end{align*}
	so that the sequence $\left(Z_i\right)_{i\geq 2}$ is a sequence of independent random variables. 
	Note that since $(1-\beta_{i-1})\sim \mathrm{Beta}(a_{i},A_{i-1}+i-1)$, we have  
	\begin{align*}
	\Ec{Z_i}
	=\frac{a_{i}\cdot (a_{i}+1)}{(A_{i}+i-1)(A_{i}+i)} 
	\leq C\cdot \frac{a_i (a_{i}+1)}{i^2}.
	\end{align*}
	This entails using Assumption \eqref{eq:assumption_2 fitness sequence} that
	\begin{align*}
	\sum_{i=n}^{\infty} \Ec{Z_i}=\grandO{n^{-1}}.
	\end{align*}
	Then, for any $n\geq 0$, let $M_n \coloneqq \sum_{i=2}^n (Z_i-\Ec{Z_i})$, which is a martingale in its own filtration. We now prove that this martingale almost surely converges to a limit $M_\infty=\sum_{i=2}^\infty (Z_i-\Ec{Z_i})$ and that we almost surely have $\abs{M_n-M_\infty}=\grandO{n^{-1}}$. 
	Together with the above, this implies that almost surely
	\begin{align*}
	\sum_{i=n}^{\infty} Z_i = \sum_{i=n}^{\infty} \Ec{Z_i}+(M_\infty-M_{n-1})= \grandO{n^{-1}},
	\end{align*}
	which is what we want to prove. 
	For this, we use \cite[Lemma~A.3]{senizergues2021} with $q=2$ and $\alpha=-1$, for which we just need to verify that 
	\begin{align*}
	\Ec{(M_{2n}-M_n)^2}\leq \grandO{n^{-2 - \delta}},
	\end{align*}
	for some $\delta>0$. 
	We have
	\begin{align*}
	\Ec{(M_{2n}-M_n)^2}=\sum_{i=n+1}^{2n}\Var(Z_i)
	\end{align*} 
	and
	\begin{align*}
	\Var(Z_i)&=\Ec{(1-\beta_{i-1})^4}-\Ec{(1-\beta_{i-1})^2}^2\\
	&=\frac{2 a_i (a_i + 1) (A_{i-1}+i-1) (2 a_i (A_i+i + 2) + 3 (A_i +i))}{(A_i+i-1)^2 (A_i+i)^2 (A_i+i+1) (A_i +i+ 2)}\\
	&\leq C\frac{(a_i+2)^3}{i^4}.
	\end{align*}
	Hence, 
	\begin{align*}
	\Ec{(M_{2n}-M_n)^2}\leq \sum_{i=n+1}^{2n}C\frac{(a_i+2)^3}{i^4} &\leq \frac{C}{n^4} \cdot \left(2+\max_{n+1\leq i \leq 2n}a_i\right)\cdot \sum_{i=n+1}^{2n} (a_i+2)^2\\
	&= C \cdot n^{-4}\cdot n^{1-\delta} \cdot C\cdot n=\grandO{n^{-2-\delta}},
	\end{align*}
	using \eqref{eq:assumption_1 fitness sequence} and \eqref{eq:assumption_2 fitness sequence}.
	This concludes the proof.
\end{proof}
\end{appendix}

\section*{Acknowledgements}
The authors would like to thank the anonymous referees for their careful reading, which helped improving the paper.

\end{document}